\documentclass[a4paper,leqno]{article}
\usepackage{amsmath,amsthm,amssymb,bbm,bm}
\usepackage{xcolor}
\usepackage{verbatim}
\usepackage{todonotes}
\usepackage{mathrsfs}

\usepackage{thmtools}
\usepackage[pdfdisplaydoctitle,colorlinks,breaklinks,urlcolor=blue,linkcolor=blue,citecolor=blue]{hyperref}
\usepackage{enumitem}
\numberwithin{equation}{section}

\newtheorem{theorem}{Theorem}[section]
\newtheorem{lemma}[theorem]{Lemma}
\newtheorem{proposition}[theorem]{Proposition}
\newtheorem{corollary}[theorem]{Corollary}
\newtheorem{definition}[theorem]{Definition}
\newtheorem{remark}[theorem]{Remark}
\newtheorem{assumption}[theorem]{Assumption}

\newcommand{\im }{\mathrm{i}}
\newcommand{\Real}{\operatorname{Re}}
\newcommand{\dualityReal}[2]{\Real \langle #1, #2 \rangle}

\begin{document}

\title{Ergodic results for the stochastic nonlinear Schr\"odinger equation with large damping}
\author{Zdzis{\l}aw Brze{\'z}niak\thanks{Department of Mathematics, University of York,
			Heslington, York, YO105DD, U \textit{E-mail address}:     zdzislaw.brzezniak@york.ac.uk}, 
		Benedetta Ferrario\thanks{Dipartimento di Scienze Economiche e Aziendali, Universit\`a di Pavia, 27100 Pavia, Italy  \textit{E-mail address}: benedetta.ferrario@unipv.it} 
		and Margherita Zanella\thanks{Dipartimento di Matematica "Francesco Brioschi", Politecnico di Milano, Via Bonardi 13, 20133 Milano, Italy \textit{E-mail address}: margherita.zanella@polimi.it	}	}

\date{\today}

\maketitle

\begin{abstract}We study the nonlinear Schr\"odinger equation with linear damping, i.e. a  zero order dissipation, and additive noise. Working in $\mathbb R^d$ with $d=2$ or $d=3$, we prove the uniqueness of the  invariant measure when the damping coefficient is sufficiently large.
\end{abstract}

\section{Introduction}
The nonlinear  Schr\"odinger equation occurs as a basic model in many areas of physics: hydrodynamics, plasma physics, optics, molecular biology, chemical reaction, etc. It describes the propagation of waves in media with both nonlinear and dispersive responses.

In this article, we investigate the long time behavior  of the nonlinear  Schr\"odinger equation
with a  linear  damping term and an additive noise. This is our model
\begin{equation}\label{eqS}
\begin{cases}
d u(t)+\left[ \im \Delta  u(t)+\im \alpha |u(t)|^{2\sigma} u(t) +\lambda u(t) \right] \,{\rm d}t 
= \Phi  \,{\rm d} W(t)
\\
u(0)=u_0,
\end{cases}
\end{equation}
The unknown is $u:\mathbb{R}^d\to\mathbb C$. We consider $\sigma>0$,   $\lambda>0$ and
$\alpha \in \{-1,1\}$;  for
$\alpha=1$ this is called the focusing equation and for $\alpha=-1$ this is  the defocusing one.

Many results are known about existence and uniqueness of solutions, 
in different spatial domains and with different noises;
see \cite{BarbuL2,BarbuH1,Barbu17,BHW-2019,Cui,DBD99,DBD03}. 
Basically these results are obtained without damping but can be easily extended 
to the case with $\lambda>0$.

When there is no damping  and no forcing term ($\lambda=0$ and $\Phi=0$), the equation is
 conservative. However, with a noise and a damping term,  we expect that 
the energy injected by the noise is dissipated by the damping term; because of this balance it is meaningful to look for  stationary solutions or invariant measures.
Ekren, Kukavica and Ziane
 \cite{Ekren_2017} and  Kim \cite{Kim} provide the existence of invariant measures of the  equation
\eqref{eqS} for any damping coefficient $\lambda>0$;
see also the more general setting of \cite{noi} for the two dimensional case 
in  a different spatial domain and with multiplicative noise and the book
\cite{Hong+Wang_2019_invariant}  for the numerical analysis approach.
Notice that the damping $\lambda u$ is weaker than the dissipation given by a  Laplacian
$-\lambda \Delta u$; 
for this reason we say that $\lambda u$ is a zero-order dissipation. 
This implies that  the results of existence or uniqueness of  invariant measures
 for the damped Schr\"odinger equation are less easy than for the stochastic 
  parabolic equations (see, e.g., \cite{DPZ2}). A similar issue appears in the 
 stochastic damped 2D Euler equations, for which the existence of invariant measures has been recently proven in 
  \cite{BesFer}; there again the difficulty comes from the absence of the strong dissipation, given by the  Laplacian in  the 
  Navier-Stokes equations.

The question of  the uniqueness of  invariant measures is quite challenging for the SPDE's with a zero-order dissipation.
Debussche and Odasso \cite{Odasso} proved the 
 uniqueness of the invariant measure for the cubic focusing
 Schr\"odinger equation \eqref{eqS}, i.e. $\sigma=\alpha=1$, when the spatial domain is a bounded interval; 
 however no uniqueness results are known for larger 
 dimension. For the one-dimensional stochastic damped KdV equation there is a recent result 
 by Glatt-Holtz,  Martinez and  Richards \cite{GHMR21}. 
  However for nonlinear SPDE's  of  parabolic type, i.e. with the stronger dissipation term, 
 the uniqueness issue has been solved in many cases; 
 see, e.g., the book \cite{DPZ2} by Da Prato and Zabczyk, and the many examples in the paper \cite{GHMR2017} by
 Glatt-Holtz, Mattingly and Richards,
 dealing with the coupling technique. Let us point out that the coupling technique allows for the 
 uniqueness result without restriction on the damping parameter $\lambda$ but all the examples 
 solved so far are set in a bounded spatial domain and not in $\mathbb R^d$.

The aim of our paper is to investigate the uniqueness of the invariant measures for equation 
\eqref{eqS} in dimension $d=2$ and $d=3$, 
with some restrictions on the nonlinearity when $d=3$. However 
our technique fails for larger dimension. 
Notice that also the results for the attractor in the deterministic setting are known for $d\le 3$ 
(see \cite{Laur}). Our main result is Theorem \ref{uniq_thm}; it provides a sufficient condition 
to get the uniqueness of the invariant measure, involving  $\lambda$  and the intensity of the noise.
 To  optimize this condition \eqref{condition_beta} we perform a 
detailed analysis on how the solution depends on  $\lambda$.

As far as the contents of this paper are concerned, in Section \ref{s-2} we introduce the mathematical setting; in Section \ref{Sect-regularity} by means of the Strichartz estimates 
we prove a regularity result on the solutions; this will allow to prove  in Section
\ref{Sect-support} that the support of any invariant measure is contained in 
$V\cap L^\infty(\mathbb R^d)$ and some estimates of the moments are given.
Finally  Section \ref{S-uniq} presents the uniqueness result. 
The four appendices contain auxiliary results.

\section{Assumptions and basic results}  \label{s-2}

For $p\ge 1$, $L^p(\mathbb R^d)$ is the classical Lebesgue space of complex valued functions, and the inner product in the real Hilbert space $L^2(\mathbb R^d)$ is denoted by 
\[
\langle u,v\rangle=\int_{\mathbb R^d} u(y)\overline v(y) dy.
\]
 We consider the Laplace operator $\Delta$ as a linear operator in $L^2(\mathbb R^d)$; so
  \[
 A_0=-\Delta,\qquad A_1=1-\Delta
 \]
  are non-negative linear operators
 and  $\{e^{\im tA_0 }\}_{t \in \mathbb R}$ is a  unitary group in $L^2(\mathbb R^d)$.
 Moreover for $s\ge 0$ we consider the  power operator  
$A_1^{s/2}$ in $L^2(\mathbb R^d)$
with domain $H^{s}=\{u \in L^2(\mathbb R^d): \|A_1^{s/2}u\|_{L^2(\mathbb R^d)}<\infty\}$.
Our two main spaces are $H:=L^2(\mathbb{R}^d)$  and $V:=H^1(\mathbb{R}^d)$.
We set $H^{-s}(\mathbb{R}^d)$ for the dual space of $H^s(\mathbb{R}^d)$ and denote 
again by  $\langle\cdot,\cdot\rangle$ the duality bracket.

We define the generalized Sobolev spaces
 $H^{s,p}(\mathbb{R}^d)$ with norm given by $\|u\|_{H^{s,p}}(\mathbb{R}^d)=\|A_1^{s/2}u\|_{L^p(\mathbb R^d)}$. 
 We recall the Sobolev embedding theorem, see e.g. \cite{Bergh}[Theorem 6.5.1]: if 
 $1<q<p<\infty$ with 
 \[
 \frac 1p=\frac 1q-\frac {r-s}d,
 \]
 then the following inclusion holds 
 \[
 H^{r,q}(\mathbb R^d)\subset H^{s,p}(\mathbb R^d)
 \]
 and there exists a constant $C$ such that $\|u\|_{H^{s,p}(\mathbb R^d)}\le C \|u\|_{H^{r,q}(\mathbb R^d)}$
 for all $u \in H^{r,q}(\mathbb R^d)$. 
 
 \begin{remark}
 For $d=1$ the space $V$ is a subset of $L^\infty(\mathbb R)$ and is a multiplicative algebra. This simplifies the analysis of the  Schr\"odinger equation \eqref{eqS}. However for $d\ge 2$
 the analysis is more involved.
 \end{remark}
 
 We write the nonlinearity as 
 \begin{equation}\label{F}
F_\alpha(u):=\alpha |u|^{2\sigma}u.
\end{equation} 
Lemma \ref{lemma_stimaF}  provides a priori estimates on it.

As far as the stochastic term is concerned, we consider  a real Hilbert space $U$ with an  orthonormal basis  $\{e_j\}_{j\in\mathbb{N}}$ and a complete probability space
 $(\Omega,\mathcal{F},\mathbb{P})$. Let $W$ be  a $U$-canonical cylindrical Wiener process adapted to a filtration $\mathbb{F}$ satisfying the usual conditions. We can write it as a series
\[
W(t)= \sum_{j=1}^\infty W_j(t) e_j, 
\]
with $\{W_j\}_j$  a sequence of i.i.d. real Wiener processes. Hence
\begin{equation}
\Phi W(t)= \sum_{j=1}^\infty W_j(t) \Phi e_j
\end{equation}
for a given linear operator  $\Phi:U\to V$.

Now we rewrite the  Schr\"odinger equation \eqref{eqS} in the abstract form as 
\begin{equation}\label{EQS}
\begin{cases}
d u(t)+\left[ -\im A_0 u(t)+\im F_\alpha( u(t)) +\lambda u(t) \right] \,{\rm d}t 
= \Phi  \,{\rm d} W(t)
\\
u(0)=u_0
\end{cases}
\end{equation}

We work under the following assumptions on the noise and the nonlinearity. 
The initial data $u_0$ is assumed to be $V$.

\begin{assumption}[on the noise] \label{ass_G}
We assume that 
 $\Phi:U\to V$ is a  Hilbert-Schmidt operator, i.e.
\begin{equation} 
\|\Phi\|_{L_{HS}(U,V)} := \big(\sum_{j=1}^\infty \|\Phi e_j\|^2_V\big)^{1/2} < \infty.
\end{equation}
\end{assumption}
This means that
\[
\|\Phi\|_{L_{HS}(U;V)}^2=\sum_{j=1}^\infty \|A_1^{1/2}\Phi e_j\|^2_H
=\sum_{j=1}^\infty  \|\Phi e_j\|^2_H +\sum_{j=1}^\infty \|\nabla \Phi e_j\|^2_H .
\]
In order to compare our setting with the more general one of our previous paper \cite{noi} 
in the two-dimensional setting, we point out that 
 $\Phi$ is also a Hilbert-Schmidt operator from $U$ to $H$ (and we denote
$\|\Phi\|_{L_{HS}(U;H)} := \left(\sum_{j \in \mathbb{N}}\|\Phi e_j\|^2_H\right)^{1/2}$) and,  
for $d=2$, a $\gamma$-radonifying operator from $U$ to  $L^{p}(\mathbb R^2)$ for any finite $p$.

\begin{assumption}[on the nonlinearity \eqref{F}]
\label{ass_lambda}
 \hfill
\begin{itemize}
\item If $\alpha=1$ (focusing), let $0\le \sigma<\frac 2d$.
\item If $\alpha=-1$ (defocusing), let $\begin{cases}0\le \sigma< \frac 2{d-2},& \text{ for }d\ge 3\\\sigma\ge 0, &  \text{ for }d\le 2\end{cases}$
\end{itemize}
\end{assumption}

We recall the  continuous embeddings
\[
\begin{split}
H^1(\mathbb R^2) \subset L^{p}(\mathbb R^2) &\qquad  \forall \ p\in [2, \infty)
\\
H^1(\mathbb R^d) \subset L^{p}(\mathbb R^d) &\qquad  \forall  \ p \in [2, \tfrac {2d}{d-2}] \text{ for } d\ge 3
\end{split}
\]
Hence for $\sigma$  chosen as in Assumption \ref{ass_lambda}
there is the continuous embedding
\begin{equation}\label{H1-e-Lsigma}
H^1(\mathbb R^d) \subset L^{2+2\sigma}(\mathbb R^d).
\end{equation} 
Moreover if $\sigma d<2(\sigma+1)$ the following Gagliardo-Niremberg inequality holds
\begin{equation}\label{GN-ineq}
\|u\|_{L^{2+2\sigma}(\mathbb R^d)}
\le C \|u\|_{L^2(\mathbb R^d)}^{1-\frac{\sigma d}{2(1+\sigma)}} \|\nabla u\|_{L^2(\mathbb R^d)}^{\frac{\sigma d}{2(1+\sigma)}}.
\end{equation}
In particular this holds for the values of $\sigma$ specified in the Assumption \ref{ass_lambda}.
In the focusing case, thanks to the 
Young inequality for any $\epsilon>0$ there exists $C_\epsilon >0$ such that
\begin{equation}\label{GN-somma}
\|u \|_{L^{2+2\sigma}(\mathbb R^d)}^{2+2\sigma}
\le
\epsilon \|\nabla u \|_{L ^2(\mathbb R^d)}^2 +C_\epsilon
 \|u \|_{L^2(\mathbb R^d)}^{2+\frac{4\sigma}{2-\sigma d}}.
\end{equation}

We recall the classical invariant quantities for the deterministic unforced  Schr\"odinger equation 
($\lambda=0$, $\Phi=0$), the mass and the energy (see \cite{Cazenave}):
\begin{align}
\label{mass}
&\mathcal{M}(u)=\|u\|^2_H,
\\
\label{energy}
&\mathcal{H}(u)=\frac 12 \|\nabla u\|^2_H -\frac{\alpha}{2(1+\sigma)}\|u\|_{L^{2+2\sigma}(\mathbb R^d)}^{2+2\sigma}.
\end{align}
They are both well defined on $V$, thanks to \eqref{H1-e-Lsigma}.
\begin{remark} In the  defocusing case $\alpha=-1$, we have 
\begin{equation}
\mathcal{H}(u)\ge  \frac 12 \|\nabla u\|^2_H\ge 0
\qquad \forall u\in V.
\end{equation}
In the focusing case $\alpha=1$, the energy has no positive sign but we can modify it by adding 
a term and recover the sign property. We introduce the modified energy
\begin{equation}\label{modifEnergy}
\tilde{\mathcal{H}}(u)= \frac 1 2 \|\nabla u\|_H^2-\frac1{2(1+\sigma)} \|u\|_{L^{2(1+\sigma)}(\mathbb R^d)}^{2+2\sigma} 
+G\|u\|_H^{2+\frac{4\sigma}{2-\sigma d}}
 \end{equation}
where $G$ is the constant appearing in the following particular form of \eqref{GN-somma}
\begin{equation}\label{constanteG}
\frac1{2(1+\sigma)} \|u\|_{L^{2(1+\sigma)}(\mathbb R^d)}^{2+2\sigma} 
\le \frac 14 \|\nabla u\|_H^2+G \|u\|_H^{2+\frac{4\sigma}{2-\sigma d}}.
\end{equation}
Therefore
\begin{equation}\label{nabla-Htilde}
 \tilde{\mathcal{H}}(u)\ge \frac 14 \|\nabla u\|_H^2\ge 0 
 \qquad \forall u\in V.
\end{equation}
\end{remark}

Above we denoted by $C$ a generic positive constant which might vary from one line to the other, 
except $G$ which is the particular constant in \eqref{constanteG} 
coming from the Gagliardo-Niremberg inequality.  
Moreover we shall use this notation: 
 if  $a,b\ge 0$ satisfy the inequality $a \le C_A b$ with a constant $C_A>0$ depending on the expression $A$, we write $a \lesssim_A b$; for a generic constant we put no subscript.
 If we have $a \lesssim_A b$ and $b \lesssim_A a$, we write $a \simeq_A  b$.
 
 Now we recall the known results on solutions and invariant measures; 
 then we provide the improved estimates for the mass and the energy.

 \subsection{Basic results} 

We  recall  from  \cite{DBD03} the basic  results on the solutions; for any $u_0\in V$ there exists a unique  global solution $u=\{u(t;u_0)\}_{t\ge 0}$, which is a   continuous $V$-valued  process.
 Here uniqueness is meant as pathwise uniqueness.  Actually their result is given without damping but one can easily pass from 
$\lambda=0$ to any $\lambda>0$. The difference consists in getting uniform in time estimates for 
the damped equation over the time interval $[0,+\infty)$, whereas they hold on any finite time 
interval $[0,T]$ when $\lambda=0$. Let us state the result from De Bouard and Debussche 
\cite{DBD03}.
\begin{theorem}
\label{existence}
Under Assumptions \ref{ass_G} and \ref{ass_lambda}, for every $u_0\in V$ 
there exists a unique $V$-valued and continuous 
solution  of \eqref{EQS}. This is a Markov process in $V$.
Moreover for any finite $T>0$  and integer $m\ge 1$ there exist  positive constants 
$C_1$ and $C_2$  (depending on $T$, $m$ and $\|u_0\|_V$) such that
\[
\mathbb E \sup_{0\le t\le T}\left[ \mathcal M(u(t))^m \right]\le C_1 
\]
and
\[
\mathbb E \left[\sup_{0\le t\le T}\mathcal H(u(t))\right] \le C_2 .
\]
\end{theorem}

We notice that  the last  estimate can be obtained for any power $m>1$ for the energy in the 
defocusing case as well as for the modified energy in the focusing case.
Similar computations
can be found in \cite{Ekren_2017} for the defocusing case; 
however their  Lemma 5.1 
has to be modified in the focusing case. 
The strategy is the same as that  given in the next Section \ref{sect-mean-estimates}.

As soon as a unique solution in $V$ is defined, we can introduce the Markov semigroup. 
Let us denote by $u(t;x)$ the solution evaluated at time $t>0$, with initial value $x$. We define
\begin{equation}
P_t f(x)=\mathbb E[f(u(t;x))]
\end{equation}
for any Borelian and bounded function $f:V\to\mathbb R$.

A probability measure $\mu$ on the Borelian subsets of $V$ is said to be an invariant measure for 
\eqref{EQS} when
\begin{equation}
\int_V P_t f\ d\mu=\int_V  f\ d\mu\qquad \forall t\ge 0, f\in B_b(V).
\end{equation}

We recall Theorem 3.4 from  \cite{Ekren_2017}. 
\begin{theorem}
\label{existence_inv_mea}
Under Assumptions \ref{ass_G} and \ref{ass_lambda}, there exists an invariant measure supported in $V$.
\end{theorem}

\subsection{Mean estimates }
\label{sect-mean-estimates}
In this section we revise some bounds on the moments of the mass,  the energy and 
the modified energy,
 in order to see how these quantities depend on the damping coefficient $\lambda$.  
 This improves the results by 
  \cite{DBD03}[\S 3] and  \cite{Ekren_2017}[Lemma 5.1].

This is the result for the mass $\mathcal{M}(u)=\|u\|^2_H$.
\begin{proposition}
\label{bound_lemma}
Let $u_0\in V$.
Then under assumptions \ref{ass_G} and \ref{ass_lambda},
for every $m\ge 1$ there exists a positive constant $C$ (depending on $m$)
such that 
\begin{equation}
\label{n_est_mass}
\mathbb{E}\left[\mathcal{M}(u(t))^m\right] 
\le
 e^{-\lambda m t}\mathcal{M}(u_0)^m
 + C \|\Phi\|^{2m}_{L_{HS}(U;H)}  \lambda^{-m}
\end{equation}
for any $t\ge 0$.
\end{proposition}
\begin{proof} 
Let us start by proving the estimate \eqref{n_est_mass} for $m=1$.
We apply the It\^o formula to  $\mathcal{M}(u(t))$  
\[
d\mathcal{M}(u(t))+ 2\lambda \mathcal{M}(u(t))dt
=\|\Phi \|^2_{L_{HS}(U;H)}dt+2\dualityReal{u(t)}{\Phi dW(t)} .
\]
Taking the expected value and using the fact that the stochastic integral is a martingale by Theorem \ref{existence}, 
we obtain, for any $t \ge 0$, 
\begin{equation*}
\frac{\rm d}{\rm dt}\mathbb{E}\left[\mathcal{M}(u(t))\right]=-2\lambda \mathbb{E}\left[\mathcal{M}(u(t))\right]+ \|\Phi \|^2_{L_{HS}(U,H)}.
\end{equation*}
Solving this ODE we obtain
\begin{align*}
\mathbb{E}\left[\mathcal{M}(u(t))\right]
&=e^{-2\lambda t}\mathcal{M}(u_0)+\|\Phi \|^2_{L_{HS}(U,H)}\int_0^t e^{-2\lambda(t-s)}\, {\rm d}s 
\\
&\le e^{-2\lambda t}\mathcal{M}(u_0)+\frac{1}{2\lambda}\|\Phi \|^2_{L_{HS}(U,H)},\end{align*}
which proves \eqref{n_est_mass} for $m=1$.

For larger values of $m$, let us  apply the It\^o formula to $\mathcal{M}(u(t))^{m}$ to obtain 
\begin{equation}\label{Ito_Mm}
  \begin{split}
\mathcal{M}(u(t))^{m}
=& \mathcal{M}(u_0)^{m}-2\lambda m \int_0^t \mathcal{M}(u(s))^{m}\, {\rm d}s
\\
&+2m\int_0^t \mathcal{M}(u(s))^{m-1} \text{Re}\langle u(s),\Phi {\rm d}W(s)\rangle
\\
&+m \|\Phi \|^2_{L_{HS}(U,H)}\int_0^t \mathcal{M}(u(s))^{m-1}\, {\rm d}s
\\
&+
2(m-1)m\int_0^t \mathcal{M}(u(s))^{m-2}\sum_{j=1}^{\infty}[\text{Re}\langle u(s),\Phi e_j \rangle]^2\,{\rm d}s.
\end{split}\end{equation}
With the Young inequality we get
\begin{equation}
\label{mass-power}
\begin{split}
m \|\Phi \|^2_{L_{HS}(U,H)}& \mathcal{M}(u)^{m-1}
+
2(m-1)m \mathcal{M}(u)^{m-2}\sum_{j=1}^{\infty}[\text{Re}\langle u,\Phi e_j \rangle]^2
\\
&\le
m (2m-1) \|\Phi \|^2_{L_{HS}(U,H)}\mathcal{M}(u)^{m-1}
\\
& \le \lambda m \mathcal{M}(u)^m 
+\|\Phi \|^{2m}_{HS(U;H)}\left(\frac {m-1}{m}\right)^{m-1} (2m-1)^m \lambda^{1-m}.
\end{split}\end{equation}
By Theorem \ref{existence} we know that  the stochastic integral in \eqref{Ito_Mm}  is a martingale, 
so taking the expected value on both sides of \eqref{Ito_Mm}
 we obtain
 \[
 \mathbb{E} \mathcal{M}(u(t))^{m}
\le \mathcal{M}(u_0)^{m}-\lambda m \int_0^t \mathbb{E} \mathcal{M}(u(s))^{m}\, {\rm d}s
+\|\Phi \|^{2m}_{HS(U;H)}C_m \lambda^{1-m} t.
 \]
  By means of Gronwall inequality we get
 \[\begin{split}
 \mathbb{E} \mathcal{M}(u(t))^{m}
 &\le  e^{-\lambda m t} \mathcal{M}(u_0)^{m}
+\|\Phi \|^{2m}_{HS(U;H)} C_m \lambda^{1-m}\int_0^t e^{-\lambda m(t-s)}ds
 \\
 &\le  e^{-\lambda m t} \mathcal{M}(u_0)^{m}
 + \|\Phi \|^{2m}_{HS(U;H)} \frac{C_m}m \lambda^{-m}
 \end{split} \]
for any  $t\ge 0$.
\end{proof}

Notice that the estimates on the mass do not depend on $\alpha$, 
whereas this happens in the following result. Indeed,  
consider  the  energy $\mathcal{H}(u)$ given in \eqref{energy} and the modified energy
$\tilde {\mathcal H}(u)$ given in \eqref{modifEnergy}. We deal with $d\ge 2 $, 
since the case $d=1$ is easier and already analysed in \cite{Odasso}. Therefore the condition 
$\sigma<\frac 2d$ implies that $\sigma<1$ when $d\ge 2$. We introduce the function
\begin{equation}
\phi_1(\sigma,\lambda,\Phi)
=
\|\Phi\|_{L_{HS}(U;V)}^{2}
+ \|\Phi\|_{L_{HS}(U;V)}^{2+2\sigma}\lambda^{-\sigma}.
\end{equation}
Notice that  the mapping 
$\lambda\mapsto \phi_1(\sigma,\lambda,\Phi)$ is strictly decreasing.
The same property holds for the other function
$ \phi_2(d, \sigma,\lambda,m, \Phi)$ appearing the next statement.
The expression of $ \phi_2(d, \sigma,\lambda,1,\Phi)$ is given in 
\eqref{phi2-m=1}; however for general $m$ the expression is similar but 
 longer and we avoid to specify it.

\begin{proposition}
\label{bound_lemma-energy}
Let $d\ge 2$ and $u_0\in V$.
Under Assumptions \ref{ass_G} and \ref{ass_lambda}, we have the following estimates:
\\{\bf i)}
 When $\alpha=-1$, for every $m\ge 1$ there exists
 a positive  constant  $C=C(d, \sigma,m)$ 
 such that
\begin{equation}
\label{n_est_energy_defoc}
\mathbb E \mathcal{H}(u(t))^m
\le 
e^{-\lambda m t}\mathcal{H}(u_0)^m+
C \phi_1^{m}\lambda^{-m}
\end{equation}
for any $t\ge 0$. 
\\{\bf ii)} When $\alpha=1$, for every  $m\ge 1$ there 
 exist a smooth positive function
$\phi_2=\phi_2(d, \sigma,\lambda,m,\Phi)$
and positive constants $a=a(d,\sigma)$,  $C_1=C(d, \sigma,m)$ and
 $C_2=C(d, \sigma,m)$ 
 such that  
\begin{multline}
\label{n_est_energy_foc}
\mathbb E \tilde {\mathcal{H}}(u(t))^m
\\\le e^{- m a\lambda t} \Big(\tilde {\mathcal{H}}(u_0)^m
+C_1(\lambda^{-m} +\lambda^{-\frac{m-1}2} )\mathcal{M}(u_0)^{m(1+\frac{2\sigma}{2-\sigma d})}\Big)
+
C_2 \phi_2^{m}\lambda^{-m}
\end{multline}
for any $t\ge 0$. Moreover the mapping 
$\lambda\mapsto \phi_2(d, \sigma,\lambda,m, \Phi)$ is strictly decreasing.
\end{proposition}
\begin{proof}
The It\^o formula for   $\mathcal{H}(u(t))$ is 
\begin{multline}\label{Ito_energy}
d\mathcal{H}(u(t))+2 \lambda \mathcal{H}(u(t)) dt=\alpha\lambda \frac \sigma{\sigma+1} \|u(t)\|_{L^{2\sigma+2}(\mathbb R^d)}^{2\sigma+2} dt
\\
-\sum_{j=1}^\infty  \Real \langle \Delta u(t)+\alpha |u(t)|^{2\sigma}u(t),\Phi e_j\rangle dW_j(t)
+\frac 12
 \|\nabla \Phi\|^2_{L_{HS}(U;H)} dt
 \\ - \frac \alpha 2 \||u(t)|^{\sigma} \Phi\|^2_{L_{HS}(U;H)}dt
-\alpha \sigma   \sum_{j=1}^\infty \langle |u(t)|^{2\sigma-2}, [\Real (\overline u(t)\Phi e_j)]^2\rangle dt.
\end{multline}
Below we repeatedly use the H\"older and Young inequalities. In particular 
\begin{equation}\label{Ym-1}
A^{m-1}B\le \epsilon \lambda A^m+C_\epsilon \lambda^{1-m} B^m
\end{equation}
and
\begin{equation}\label{Ym-2}
A^{m-2}B\le \epsilon \lambda A^m+C_\epsilon \lambda^{1-\frac m2} B^{\frac m2}
\end{equation}
for positive $A,B,\lambda, \epsilon$.

\noindent 
$\bullet$ In the defocusing case $\alpha=-1$,  we neglect the first term in the r.h.s. in \eqref{Ito_energy}, i.e.
\begin{multline}\label{Ito_energy_defocusing}
d\mathcal{H}(u(t))+2 \lambda \mathcal{H}(u(t)) dt\le
-\sum_{j=1}^\infty  \Real \langle \Delta u(t)- |u(t)|^{2\sigma}u(t),\Phi e_j\rangle dW_j(t)
\\+\Big[\frac 12
 \|\nabla \Phi\|^2_{L_{HS}(U;H)}  +
   \frac 1 2 \||u(t)|^{\sigma} \Phi\|^2_{L_{HS}(U;H)}
+ \sigma   \sum_{j=1}^\infty \langle |u(t)|^{2\sigma-2}, [\Real (\overline u(t)\Phi e_j)]^2\rangle\Big]dt.
  \end{multline}
Moreover thanks to the Assumption \ref{ass_lambda} we  use  the
 H\"older and Young inequalities to   get
\begin{equation}\label{u-Phi_HS}
\begin{split}
\frac 1 2 & \||u|^{\sigma} \Phi\|^2_{L_{HS}(U;H)}
+ \sigma \sum_{j=1}^\infty \langle |u|^{2\sigma-2}, [\Real (\overline u(t)\Phi e_j)]^2\rangle
\\&\le
\frac 12 \||u|^\sigma\|^2_{L^{\frac{2\sigma+2}\sigma}(\mathbb R^d)}\sum_{j=1}^\infty \|\Phi e_j\|_{L^{2\sigma+2}(\mathbb R^d)}^2+\sigma \||u|^{2\sigma}\|_{L^{\frac{2\sigma+2}{2\sigma}}(\mathbb R^d)}\sum_{j=1}^\infty \||\Phi e_j|^2\|_{L^{\sigma+1}(\mathbb R^d)}
\\&
\le\frac{2\sigma+1}2 \|u \|_{L^{2\sigma+2}(\mathbb R^d)}^{2\sigma} 
\sum_{j=1}^\infty \| \Phi e_j\|_{L^{2\sigma+2}(\mathbb R^d)}^2
\\
&
\le \frac{2\sigma+1}2 \|u \|_{L^{2\sigma+2}(\mathbb R^d)}^{2\sigma}  \|\Phi\|_{L_{HS}(U;V)}^2
\quad\text{ by } \eqref{H1-e-Lsigma}
\\&\le
\frac\lambda{2+2\sigma} \|u\|_{L^{2+2\sigma}(\mathbb R^d)}^{2+2\sigma} 
+C_\sigma \| \Phi \|^{2+2\sigma}_{L_{HS}(U;V)} \lambda^{-\sigma}
\\&\le
\lambda \mathcal H(u)+C\| \Phi \|^{2+2\sigma}_{L_{HS}(U;V)} \lambda^{-\sigma}
\end{split}\end{equation}
Now we insert this estimate in  \eqref{Ito_energy_defocusing} and 
take  the mathematical expectation to get rid of the stochastic integral
\[
\frac{d}{dt}\mathbb E \mathcal{H}(u(t))+2\lambda \mathbb E\mathcal{H}(u(t))
\le \frac 12 \| \Phi\|^2_{L_{HS}(U;V)}+\lambda  \mathbb E \mathcal H(u(t))
+C \| \Phi \|^{2+2\sigma}_{L_{HS}(U;V)} \lambda^{-\sigma},
\]
i.e.
\[
\frac{d}{dt}\mathbb E \mathcal{H}(u(t))+\lambda \mathbb E\mathcal{H}(u(t))
\le  \frac 12 \| \Phi\|^2_{L_{HS}(U;V)}+
C \| \Phi \|^{2+2\sigma}_{L_{HS}(U;V)} \lambda^{-\sigma} .
\]
By Gronwall lemma we get
\[
\mathbb E \mathcal{H}(u(t))\le e^{-\lambda t}\mathcal{H}(u_0)+
 \frac 1{2} \| \Phi\|^2_{L_{HS}(U;V)}\lambda^{-1}+C \| \Phi \|^{2+2\sigma}_{L_{HS}(U;V)} \lambda^{-\sigma-1} 
\]
for any $t\ge 0$. This proves \eqref{n_est_energy_defoc} for $m=1$.

For higher powers $m >1$, by means of It\^o formula we get
\begin{multline}\label{Ito-H-alla-m}
d \mathcal{H}(u(t))^m=m \mathcal{H}(u(t))^{m-1} d \mathcal{H}(u(t))
\\
+\frac{m(m-1)}2 \mathcal{H}(u(t))^{m-2} \sum_{j=1}^\infty
  [\Real \langle \Delta u(t)-|u(t)|^{2\sigma}u(t),\Phi e_j\rangle]^2 dt.
\end{multline}
 We estimate the latter term using H\"older and 
 Young inequality:
\[\begin{split}
\frac 12 \sum_j & [\Real \langle \Delta u-|u|^{2\sigma}u,\Phi e_j\rangle]^2
\\&\le
 \sum_j  [\Real \langle \Delta u,\Phi e_j\rangle]^2
+\sum_j  [\Real \langle |u|^{2\sigma}u,\Phi e_j\rangle]^2
\\&
\le \|\nabla u\|_{H}^2  \sum_j \|\nabla \Phi e_j\|_{H}^2
+ \||u|^{2\sigma}u\|^2_{L^{\frac{2\sigma+2}{2\sigma+1}}(\mathbb R^d)} \sum_j \|\Phi e_j\|^2_{L^{2+2\sigma}(\mathbb R^d)}
\\&
\le  \|\nabla u\|_{H}^2  \|\Phi\|_{L_{HS}(U;V)}^2
+ \|u\|^{2(2\sigma+1)}_{L^{2+2\sigma}(\mathbb R^d)} \|\Phi\|_{L_{HS}(U;V)}^2
\\&
\le \epsilon \lambda \mathcal{H}(u)^2+ C_{\epsilon,\sigma}\left(\|\Phi\|_{L_{HS}(U;V)}^{4}
+\|\Phi\|_{L_{HS}(U;V)}^{4(1+\sigma)}\lambda^{-2\sigma}\right)\lambda^{-1}
\end{split}\]
for any $\epsilon>0$. 
Inserting in \eqref{Ito-H-alla-m}  and using the
Young inequality \eqref{Ym-2}
 we get
\begin{multline}\label{ito_m}
d \mathcal{H}(u(t))^m \le m \mathcal{H}(u(t))^{m-1} d \mathcal{H}(u(t))
\\+\frac 12 m\lambda \mathcal{H}(u(t))^m dt 
+C \left(\|\Phi\|_{L_{HS}(U;V)}^{4}
+\|\Phi\|_{L_{HS}(U;V)}^{4(\sigma+1)}\lambda^{-2\sigma}\right)^{m/2}\lambda^{-m+1}
dt
\end{multline}

We estimate $\mathcal{H}(u(t))^{m-1} d \mathcal{H}(u(t))$ using \eqref{Ito_energy_defocusing}, 
 \eqref{u-Phi_HS},  and the Young inequality \eqref{Ym-1}.
Then we take the mathematical expectation  in \eqref{ito_m} and  obtain
\begin{multline}
\frac{d}{dt}\mathbb E \mathcal{H}(u(t))^m+m \lambda \mathbb E \mathcal{H}(u(t))^m 
\\
\le
 C_{\sigma,m}\left(\|\Phi\|_{L_{HS}(U;V)}^{2}
+ \|\Phi\|_{L_{HS}(U;V)}^{2(1+\sigma)}\lambda^{-\sigma}\right)^{m}\lambda^{-m+1}.
\end{multline}
By Gronwall lemma we get \eqref{n_est_energy_defoc}.

\noindent
$\bullet$ In the focusing  case $\alpha=1$, we neglect the last two  terms in the r.h.s. in \eqref{Ito_energy} and get
\begin{multline}
d\mathcal{H}(u(t))+2 \lambda \mathcal{H}(u(t)) dt
\le \lambda \frac \sigma{\sigma+1} \|u(t)\|_{L^{2+2\sigma}(\mathbb R^d)}^{2+2\sigma} dt
\\
-\sum_{j=1}^\infty  \Real \langle\Delta u(t)+ |u(t)|^{2\sigma}u(t),\Phi e_j\rangle dW_j(t)
+\frac 12
 \|\nabla \Phi\|^2_{L_{HS}(U;H)} dt.
\end{multline}
We write the It\^o formula for the modified energy 
$\tilde{\mathcal{H}}(u)= {\mathcal{H}}(u)
+G\mathcal M(u)^{1+\frac{2\sigma}{2-\sigma d}}$.
Keeping in mind \eqref{Ito_Mm} and \eqref{mass-power} for the power 
$m=1+\frac{2\sigma}{2-\sigma d}$ of the  mass, 
we have
\begin{multline}
d\tilde{\mathcal{H}}(u(t))+2 \lambda \tilde{\mathcal{H}}(u(t)) dt
\\
\le\lambda \frac \sigma{\sigma+1} \|u(t)\|_{L^{2\sigma+2}(\mathbb R^d)}^{2\sigma+2} dt
-2\lambda \frac{2\sigma}{2-\sigma d} G \mathcal M(u(t))^{1+\frac{2\sigma}{2-\sigma d}} dt\;
\\
\qquad +C \mathcal M(u(t))^{1+\frac{2\sigma}{2-\sigma d}} dt
 +C \|\Phi \|^{2+\frac{4\sigma}{2-\sigma d}}_{L_{HS}(U;H)} dt +\frac 12
 \|\nabla \Phi\|^2_{L_{HS}(U;H)} dt
\\
-\sum_j  \Real \langle \Delta u(t)+ |u(t)|^{2\sigma}u(t),\Phi e_j\rangle dW_j(t)
\\
+2(1+\frac{2\sigma}{2-\sigma d}) G \mathcal{M}(u(s))^{\frac{2\sigma}{2-\sigma d}} \text{Re}\langle u(t),\Phi {\rm d}W(t)\rangle.
\end{multline}
Since $(1-\frac2{2-\sigma d})\le 0$ by Assumption \ref{ass_lambda}, we get
\[\begin{split}
 \frac \sigma{\sigma+1} \|u\|_{L^{2\sigma+2}(\mathbb R^d)}^{2\sigma+2} 
&-\frac{4\sigma}{2-\sigma d} G \mathcal M(u)^{1+\frac{2\sigma}{2-\sigma d}} 
\\&\underset{\eqref{constanteG}}{\le}
\frac \sigma 2 \|\nabla u\|_H^2+2\sigma (1-\frac2{2-\sigma d})G \mathcal M(u)^{1+\frac{2\sigma}{2-\sigma d}}
\\&
\le\frac \sigma 2 \|\nabla u\|_H^2 
\underset{\eqref{nabla-Htilde}}{\le} 2\sigma \tilde{\mathcal H}(u).
\end{split}
\]
Then
\begin{multline}
d\tilde{\mathcal{H}}(u(t))+2 (1-\sigma)\lambda \tilde{\mathcal{H}}(u(t)) dt
\\
\le \Big(C \mathcal M(u(t))^{1+\frac{2\sigma}{2-\sigma d}} 
 +C \|\Phi \|^{2+\frac{4\sigma}{2-\sigma d}}_{L_{HS}(U;H)}  +\frac 12
 \|\nabla \Phi\|^2_{L_{HS}(U;H)}\Big) dt
\\
-\sum_j  \Real \langle \Delta u(t)+ |u(t)|^{2\sigma}u(t),\Phi e_j\rangle dW_j(t)
\\
+2(1+\frac{2\sigma}{2-\sigma d}) G \mathcal{M}(u(s))^{\frac{2\sigma}{2-\sigma d}} 
\Real \langle u(t),\Phi {\rm d}W(t)\rangle.
\end{multline}
So
considering the mathematical expectation we obtain
\begin{equation}
\begin{split}
\frac{d}{dt}\mathbb E \tilde{\mathcal{H}}(u(t))&
+2(1-\sigma)\lambda \mathbb E\tilde{\mathcal{H}}(u(t))
\\&
\le C\mathbb E[\mathcal M(u(t))^{1+\frac{2\sigma}{2-\sigma d}} ]
   +C \|\Phi \|^{2+\frac{4\sigma}{2-\sigma d}}_{L_{HS}(U;H)} 
   +\frac 12 \|\nabla \Phi\|^2_{L_{HS}(U;H)}
\\&
\le Ce^{-\lambda (1+\frac{2\sigma}{2-\sigma d})t}\mathcal M(u_0)^{1+\frac{2\sigma}{2-\sigma d}} 
 + C \|\Phi \|_{L_{HS}(U;H)}^{2 (1+\frac{2\sigma}{2-\sigma d})}\lambda^{- (1+\frac{2\sigma}{2-\sigma d})}
\\&
\quad +C \|\Phi \|^{2+\frac{4\sigma}{2-\sigma d}}_{L_{HS}(U;H)} 
   +\frac 12 \|\nabla \Phi\|^2_{L_{HS}(U;H)}
\end{split}\end{equation}
thanks to \eqref{n_est_mass}.
By means of the Gronwall lemma, 
setting $a=\min(2-2\sigma, 1+\frac{2\sigma}{2-\sigma d})>0$ 
  we obtain
\[
\mathbb E \tilde{\mathcal{H}}(u(t))
\le
e^{-2(1-\sigma)\lambda t} \tilde{\mathcal{H}}(u_0)
+e^{- a\lambda t}\lambda^{-1} C \mathcal{M}(u_0)^{1+\frac{2\sigma}{2-\sigma d}}
+C \phi_2\lambda^{-1} 
\]
where $\phi_2=\phi_2(d, \sigma,\lambda,1, \Phi) $ is equal to 
\begin{equation}\label{phi2-m=1}
\|\Phi\|_{L_{HS}(U;H)}^{2(1+\frac{2\sigma}{2-\sigma d})}\left(\lambda^{-1-\frac{2\sigma}{2-\sigma d}}+1\right)
+\|\nabla \Phi\|^2_{L_{HS}(U;H)}.\end{equation}
This proves \eqref{n_est_energy_foc} for $m=1$.

For   $m >1$, we have by It\^o formula 
\begin{equation}\label{Ito-tildeH-alla-m}
d \tilde{\mathcal{H}}(u(t))^m\le m \tilde{\mathcal{H}}(u(t))^{m-1} d \tilde{\mathcal{H}}(u(t))
+\frac{m(m-1)}2 
\tilde{\mathcal{H}}(u(t))^{m-2}
2r(t)\ dt,
\end{equation}
where
\[
  r(t)=\sum_{j=1}^{\infty}  [\Real \langle \Delta u(t)+|u(t)|^{2\sigma}u(t),\Phi e_j\rangle ]^2 
    +  4G^2(1+\tfrac{2\sigma}{2-\sigma d})^2
    \mathcal M(u(t))^{\frac{4\sigma}{2-\sigma d}}\sum_{j=1}^{\infty}[\text{Re}\langle u(t),\Phi e_j \rangle]^2 .
  \]
Keeping in mind the previous estimates we get 
\[\begin{split}
 r(t)\lesssim & \|\nabla u(t)\|_H^2 \|\Phi\|^2_{L_{HS}(U;V)}
+ \|u(t)\|^{2(2\sigma+1)}_{L^{2\sigma+2}(\mathbb R^d)} \|\Phi\|_{L_{HS}(U;V)}^2
\\
&
+ 4G^2(1+\tfrac{2\sigma}{2-\sigma d})^2
    \mathcal M(u(t))^{1+\frac{4\sigma}{2-\sigma d}}   \|\Phi\|_{L_{HS}(U;H)}^2.
\end{split}\]
Now we use \eqref{nabla-Htilde}, i.e. $\|\nabla u(t)\|_H^2 \le 4  \tilde{\mathcal{H}}(u)$, and 
 by means of \eqref{GN-somma} we get
\[\begin{split}
 \|u\|^{2(2\sigma+1)}_{L^{2\sigma+2}(\mathbb R^d)} 
 &\le
\frac \epsilon 4 \|\nabla u\|_H^{2\frac{2\sigma+1}{\sigma+1}}
  +C_{\epsilon,\sigma}\mathcal M(u)^{\frac{2\sigma+1}{\sigma+1}(1+\frac{2\sigma}{2-\sigma d})}        \\
 &\le  \epsilon
        \tilde{\mathcal{H}}(u)^{\frac{2\sigma+1}{\sigma+1}}+C_{\epsilon,\sigma} \mathcal M(u)^{\frac{2\sigma+1}{\sigma+1}(1+\frac{2\sigma}{2-\sigma d})}
\end{split}\]
for any $\epsilon>0$. Thus
\[\begin{split}
\tilde{\mathcal{H}}(u(t))^{m-2} r(t)
\lesssim& \
\tilde{\mathcal{H}}(u(t))^{m-1} \|\Phi\|^2_{L_{HS}(U;V)}
+
\tilde{\mathcal{H}}(u(t))^{m-\frac1{\sigma+1}} \|\Phi\|^2_{L_{HS}(U;V)}
\\&+
\tilde{\mathcal{H}}(u(t))^{m-2} \mathcal M(u)^{\frac{2\sigma+1}{\sigma+1}(1+\frac{2\sigma}{2-\sigma d})}
\|\Phi\|^2_{L_{HS}(U;V)}
\\&+
\tilde{\mathcal{H}}(u(t))^{m-2}  \mathcal M(u(t))^{1+\frac{4\sigma}{2-\sigma d}}   \|\Phi\|_{L_{HS}(U;H)}^2.
\end{split}\]

In \eqref{Ito-tildeH-alla-m} we insert this estimate and the previous estimates for $d \tilde{\mathcal{H}}(u(t))$,  integrate in time, take the 
mathematical expectation to get rid of the stochastic integrals  and use Young inequality; hence we obtain
\[\begin{split}
\frac d{dt} &\mathbb E [\tilde{\mathcal H}(u(t))^m]
+2m(1-\sigma) \lambda 
\mathbb E [\tilde{\mathcal H}(u(t))^m]
\le
m (1-\sigma) \lambda \mathbb E [\tilde{\mathcal H}(u(t))^m]
\\&
+ C \lambda^{1-m} \mathbb E [\mathcal M(u(t))^{m(1+\frac{2\sigma}{2-\sigma d})}]
\\&
+ C \lambda^{1-\frac m2}\left(\|\Phi\|_{L_{HS}(U;V)}^m 
\mathbb E[ \mathcal M(u(t))^{\frac m2 \frac {2\sigma+1}{\sigma+1}(1+\frac{2\sigma}{2-\sigma d})} ]
+\|\Phi\|_{L_{HS}(U;H)}^m  \mathbb E[ \mathcal M(u(t))^{\frac m2 (1+\frac{4\sigma}{2-\sigma d})} ] \right)\\
&
+C \left(  \|\Phi\|^{2}_{L_{HS}(U;V)}+\|\Phi\|^{2+\frac{4\sigma}{2-\sigma d}}_{L_{HS}(U;H)} \right)^m\lambda^{1-m}
+\|\Phi\|^{2m(\sigma+1)}_{L_{HS}(U;V)} \lambda^{1-m(\sigma+1)} .
\end{split}\]
Bearing in mind the estimates \eqref{n_est_mass} for the mass, we get
an inequality for $\mathbb E [\tilde{\mathcal H}(u(t))^m]$ thanks to  Gronwall lemma. Then  a 
repeated use of the Young inequality with long but elementary calculations provides 
 that there exist  a constant $a=a(d,\sigma)>0$  and a smooth function
$\phi_2=\phi_2(d, \sigma,\lambda,m, \Phi)$,  strictly decreasing
w.r.t. $\lambda$, such that
\[
\mathbb E \tilde{\mathcal H}(u(t))^m
\le
e^{-m a \lambda t}  \Big(\tilde{\mathcal H}(u_0)^m+ 
C (\lambda^{-m} +\lambda^{-\frac{m-1}2}) \mathcal M(u_0)^{m(1+\frac{2\sigma}{2-\sigma d})} 
\Big)
+C \phi_2^m \lambda^{-m}
\]
for any $t\ge 0$.
\end{proof}

Merging the results for the mass and the energy, we obtain the result for the $V$-norm.
Indeed $\| u\|_V^2=\|\nabla u\|_H^2+\|u\|_H^2$ and 
\[
\|\nabla u\|_H^2
=
2 \mathcal H(u)+\frac \alpha{\sigma+1}  \| u\|_{L^{2\sigma+2}(\mathbb R^d)}^{2\sigma+2} .
\]
For $\alpha=-1$ we trivially get
\[
\|u\|_V^2 \le 2 \mathcal H(u)+ \mathcal M(u).
\]
For  $\alpha=1$, 
we have from \eqref{nabla-Htilde}
\[
\|u\|_V^2 \le 4\tilde{ \mathcal H}(u)+  \mathcal M(u).
\]
Here
$\phi_1$ and $\phi_2$ are the functions appearing in Proposition \ref{bound_lemma-energy}.

\begin{corollary}\label{corollario-normaV}
Let $d\ge 2$ and $u_0\in V$.
Under Assumptions \ref{ass_G} and \ref{ass_lambda}, for every $m\ge 1$ we have the following estimates:
\\
{\bf i)}  when $\alpha=-1$
\begin{equation}\label{defocusing-V-2m}
\mathbb{E}[ \|u(t)\|_V^{2m}]
\lesssim 
e^{-m \lambda  t}[{\mathcal{H}}(u_0)^m
+{\mathcal{M}}(u_0)^m]+ [\phi_1 + \|\Phi\|^{2}_{L_{HS}(U;H)}]^m \lambda^{-m} 
\end{equation}
for any $t\ge 0$;
\\ {\bf ii)} when $\alpha=1$, 
there  is a   positive constant $a=a(d,\sigma)$ such that
\begin{multline}\label{focusing-V-2m}
\mathbb{E}[ \|u(t)\|_V^{2m}]
\lesssim 
e^{-m a \lambda  t}[\tilde {\mathcal{H}}(u_0)^m
+ (\lambda^{-m} +\lambda^{-\frac{m-1}2})\mathcal{M}(u_0)^{m(1+\frac{2\sigma}{2-\sigma d})}
+{\mathcal{M}}(u_0)^m]
\\+ [\phi_2 + \|\Phi\|^{2}_{L_{HS}(U;H)}]^m \lambda^{-m} 
\end{multline}
for any $t\ge 0$. 
\end{corollary}
The constant providing the  above estimates $\lesssim $  depends on $m,\sigma$ and $d$ but not on  $\lambda$.

\section{Regularity results for  the solution}
\label{Sect-regularity}

For the solution of equation  \eqref{EQS}    
we know that  $u\in C([0,+\infty);V)$ a.s. if $u_0\in V$. 
Now we look for the $L^\infty(\mathbb R^d)$-space regularity of the paths.
 When $d=1$, this follows directly from the Sobolev embedding $H^1(\mathbb R)\subset L^\infty(\mathbb R)$. 
 But such an embedding does not hold  for $d>1$.
 However for $d=2$ or $d=3$ one can obtain  the $L^\infty(\mathbb R^d)$-regularity
 by means of the  deterministic and stochastic Strichartz estimates. 
 
Let $\phi_1$ and $\phi_2$ be the functions appearing in Proposition \ref{bound_lemma-energy}.
\begin{proposition} \label{prop-media-d2}
Let $d=2$ or $d=3$. In addition to the Assumptions \ref{ass_G} and \ref{ass_lambda}
we suppose that $\sigma<\frac{1+\sqrt {17}}4$ when $d=3$.

Given any finite $T>0$ and $u_0\in V$  the solution of equation \eqref{EQS} is in 
$ L^{2\sigma}(\Omega;L^{2\sigma}(0,T;L^\infty(\mathbb R^d)))$.
Moreover there exist   positive constants $b_1=b_1(\sigma)$,  $b_2=b_2(\sigma)$ and  
$C=C(\sigma,d, T)$  such that 
\begin{multline}\label{Goubet_est}
\mathbb E \|u\|^{2\sigma}_{L^{2\sigma}(0,T;L^\infty(\mathbb R^d))} 
\\
\le C \left(
\|u_0\|_V^{2\sigma}+ \lambda^{-b_1} 
   \psi(u_0)^{\sigma(2\sigma+1)}
  +\phi_3^{\sigma(2\sigma+1)}\lambda^{-{\sigma(2\sigma+1)}} 
  + \|\Phi\|_{L_{HS}(U;V)}^{2\sigma} \right).
\end{multline}
where
\begin{equation}\label{psi-u-0}
\psi(u_0)= \begin{cases}
{\mathcal{H}}(u_0)+{\mathcal{M}}(u_0), & \alpha=-1\\
\tilde {\mathcal{H}}(u_0)
    +(\lambda^{-1} +\lambda^{-b_2})\mathcal{M}(u_0)^{1+\frac{2\sigma}{2-\sigma d}} 
+{\mathcal{M}}(u_0),&  \alpha=1
\end{cases}
\end{equation}
    and
\begin{equation}
\phi_3(d, \sigma,\lambda, \Phi)=
\begin{cases}
\phi_1(\sigma,\lambda, \Phi) + \|\Phi\|^2_{L_{HS}(U;H)},&\alpha=-1\\
 \phi_2(d,\sigma,\lambda,\sigma(2\sigma+1), \Phi) + \|\Phi\|^2_{L_{HS}(U;H)},&\alpha=1\end{cases}
\end{equation}
 so $\lambda \mapsto \phi_3(d, \sigma,\lambda, \Phi)$ is a strictly decreasing function.
\end{proposition}
\begin{proof}
First let us consider  $d=2$.
We  repeatedly use the embedding $H^{1,q}(\mathbb R^2)\subset L^\infty(\mathbb R^2)$ 
valid for any  $q>2$. So our target is to prove the estimate for the 
$L^{2\sigma}(\Omega;L^{2\sigma}(0,T;H^{1,q}(\mathbb R^2)))$-norm of $u$
for some $q>2$.
 
We introduce the operator $\Lambda:=-iA_0+\lambda$. 
It generates the semigroup $e^{-i\Lambda t}= e^{-\lambda t} e^{iA_0 t}$, $t\ge 0$.

Let us fix $T>0$.
We 
write equation \eqref{EQS} in the mild form (see \cite{DBD03})
\begin{align}
\label{eq1}
\im u(t)
&= \im e^{- \Lambda t}u_0+ \int_0^{t} e^{-\Lambda(t-s)}F_\alpha(u(s))\, {\rm d}s
+ \im \int_0^{t} e^{-\Lambda(t-s)}\Phi \, {\rm d}W(s)
\notag\\
&=:I_1(t)+I_2(t)+I_3(t)
\end{align}
and estimate 
\[
\mathbb E \|I_i\|^{2\sigma}_{L^{2\sigma}(0,T;H^{1,q}(\mathbb R^2))},\qquad i=1,2,3
\]
for some $q>2$.

For the estimate of $I_1$ we set
\begin{equation}\label{q-per-I1}
q=\begin{cases} 
\frac {2\sigma}{\sigma-1}& \text{ if } \sigma>1\\
\frac 6{3-\sigma} &  \text{ if } 0<\sigma \le 1
\end{cases}
\end{equation}
Notice that $q>2$.  Now, before using the homogeneous Strichartz inequality \eqref{hom_Str} we neglect the term 
$e^{-\lambda t}$, since  $e^{-\lambda t} \le 1$. First, assuming  $\sigma>1$ we 
work with the admissible Strichartz pair $(2\sigma,\frac{2\sigma}{\sigma-1})$ and get
\begin{align*}
 \|I_1\|_{L^{2\sigma}(0,T;H^{1,\frac{2\sigma}{\sigma-1}}(\mathbb R^2))}
&= \left \Vert e^{-\lambda \cdot}e^{iA_0 \cdot}A_1^{1/2} u_0\right\Vert_{L^{2\sigma}(0,T;L^{\frac{2\sigma}{\sigma-1}}(\mathbb R^2))} 
\\
&\le  \left \Vert e^{iA_0 \cdot}A_1^{1/2} u_0\right\Vert_{L^{2\sigma}(0,T;L^{\frac{2\sigma}{\sigma-1}}(\mathbb R^2))} 
\\
&\lesssim  \| A_1^{1/2}  u_0\|_{L^2(\mathbb R^2)} =\|u_0\|_V
\end{align*}
For smaller values, i.e. $0<\sigma \le 1$, we choose $\tilde \sigma=\frac 3\sigma>2>\sigma$ so 
$\frac {2\tilde\sigma}{\tilde\sigma-1}=\frac 6{3-\sigma}$ and
\[
 \|I_1\|_{L^{2\sigma}(0,T;H^{1,\frac{2\tilde\sigma}{\tilde\sigma-1}}(\mathbb R^2))}
 \lesssim
  \|I_1\|_{L^{2\tilde \sigma}(0,T;H^{1,\frac{2\tilde\sigma}{\tilde\sigma-1}}(\mathbb R^2))}
  \lesssim \|u_0\|_V
\]
by the previous computations.

For the estimate of $ I_2$, 
we use the  Strichartz inequality \eqref{hom_Str_f}
and then the estimate from Lemma \ref{lemma_stimaF} on the nonlinearity.  We bear in mind the notation $\gamma^\prime$  for the conjugate exponent of  $\gamma\in (1,\infty)$, i.e. 
$\frac 1\gamma+\frac1{\gamma^\prime}=1$.
First, consider $\sigma>1$;  the pair $(2\sigma,\frac{2\sigma}{\sigma-1})$ is  admissible. Then
\[\begin{split}
\|I_2\|_{L^{2\sigma}(0,T;H^{1,\frac{2\sigma}{\sigma-1}}(\mathbb R^2))}
&=
\|A_1^{1/2}I_2\|_{L^{2\sigma}(0,T;L^{\frac{2\sigma}{\sigma-1}}(\mathbb R^2))}
\\&\lesssim
\|A_1^{1/2}F_\alpha(u)\|_{L^{\frac 43}(0,T;L^{\frac 43}(\mathbb R^2))} \quad\text{ by } \eqref{hom_Str_f}
\\&=
\|F_\alpha(u)\|_{L^{\frac 43}(0,T;H^{1,\frac 43}(\mathbb R^2))} 
\\&\lesssim
\|u\|^{2\sigma+1}_{L^{\frac 43(2\sigma+1)}(0,T;V)}
\quad\text{ by } \eqref{d2-F(u)} \text{ and } \eqref{stimaF_d2}
 \end{split}\]
For $0<\sigma\le 1$ we proceed in a similar way; considering the  admissible Strichartz pair
$(2+\sigma, 2+\frac 4\sigma)$ we have
\[\begin{split}
\|I_2\|_{L^{2\sigma}(0,T;H^{1,2+\frac 4\sigma}(\mathbb R^2))}
&\lesssim 
\|I_2\|_{L^{2+\sigma}(0,T;H^{1,2+\frac 4\sigma}(\mathbb R^2))}
\\&=
\|A_1^{1/2}I_2\|_{L^{2+\sigma}(0,T;L^{2+\frac 4\sigma}(\mathbb R^2))}
\\&\lesssim
\|A_1^{1/2}F_\alpha(u)\|_{L^{\gamma^\prime}(0,T;L^{r^\prime}(\mathbb R^2))} \quad\text{ by } \eqref{hom_Str_f}
\\&=
\|F_\alpha(u)\|_{L^{\gamma^\prime}(0,T;H^{1,r^\prime}(\mathbb R^2))} 
 \end{split}\]
where  $(r,\gamma)$ is an admissible  Strichartz pair. 
According to \eqref{d2-F(u)} we choose
\begin{equation}
(1,2)\ni r^\prime=\begin{cases}\frac 2{1+2\sigma}, &0<\sigma<\frac 12\\ \frac 43, & \frac 12 \le \sigma\le1  \end{cases}
\end{equation}
Hence
\begin{equation}
\gamma^\prime=\tfrac{2r^\prime}{3r^\prime-2}=\begin{cases}\frac1{1-\sigma},& 0<\sigma<\frac 12\\ \frac 43, & \frac 12 \le \sigma\le1 \end{cases}
\end{equation}

In this way by means of the estimate \eqref{stimaF_d2} of the polynomial nonlinearity
$\|F_\alpha(u))\|_{H^{1,r^\prime}(\mathbb R^2)}\lesssim \|u\|^{1+2\sigma}_V$ we obtain 
\[
\|I_2\|_{L^{2\sigma}(0,T;H^{1,2+\frac 4\sigma}(\mathbb R^2))}
\lesssim
\|u\|^{2\sigma+1}_{L^{\gamma^\prime(2\sigma+1)}(0,T;V)}.
\]

Summing up,  we have shown that for any $\sigma>0$ there exists $q>2$ and $\gamma^\prime$
such that 
  \begin{equation}\label{stima-I2-d2}
 \mathbb E   \|I_2\|^{2\sigma}_{L^{2\sigma}(0,T;H^{1,q}(\mathbb R^2))}
\lesssim
 \mathbb E\left( \int_0^T \|u(t)\|^{\gamma^\prime(2\sigma+1)}_{V}\ {\rm d}t\right)^{\frac {2\sigma} {\gamma^\prime}}.
  \end{equation}
Bearing in mind Corollary \ref{corollario-normaV}, we get the second and third terms in the r.h.s. of 
\eqref{Goubet_est}.
The details are given in the Appendix \ref{s-conti}.

It remains to estimate the term $ I_3$.  We choose $q$ as in \eqref{q-per-I1}.
Using the stochastic Strichartz estimate \eqref{eqn-Strichartz-213}, we get for $\sigma>1$
\begin{align*}
\mathbb E \|I_3\|^{2\sigma}_{L^{2\sigma}(0,T; H^{1,\frac{2\sigma}{\sigma-1}}(\mathbb R^2))}
&= 
\mathbb E \|A_1^{1/2} I_3\|^{2\sigma}_{L^{2\sigma}(0,T; L^{\frac{2\sigma}{\sigma-1}}(\mathbb R^2))}
\\
&\lesssim \| A_1^{1/2}\Phi\|_{L_{HS}(U;H)}^{2\sigma}
= \|\Phi\|_{L_{HS}(H;V)}^{2\sigma}.
\end{align*}
For smaller values of $\sigma$,  we proceed as before for $I_1$.

Now consider  $d=3$. The additional assumption on $\sigma$ appears because of the 
stronger conditions on the parameters given later on. 
\\
For $q\ge 1$ we have $H^{\theta,q}(\mathbb R^3)\subset L^\infty(\mathbb R^3)$ when $\theta q>3$. 
So for each $I_i$ in \eqref{eq1} 
we look for an estimate in the norm $L^{2\sigma}(0,T; H^{\theta,q}(\mathbb R^3))$ for some 
parameters $\theta q>3$.

We estimate $ I_1$ for any $0<\sigma<2$. When  $0<\sigma \le 1$ we consider the admissible Strichartz pair $(2,6)$.  
By means of the homogeneous Strichartz estimate \eqref{hom_Str} we proceed as before
\begin{align*}
 \| I_1\|_{L^{2\sigma}(0,T; H^{1,6 } (\mathbb R^3))}
 &\lesssim
 \| I_1\|_{L^{2}(0,T; H^{1,6 } (\mathbb R^3))}
 \\
&=  \left \Vert 
e^{-\lambda \cdot}e^{iA_0 \cdot} A_1^{1/2}u_0\right\Vert_{L^{2}(0,T;L^{6}(\mathbb R^3))} 
\\
&\le \left \Vert e^{i A_0 \cdot} A_1^{1/2}u_0\right\Vert_{L^{2}(0,T;L^{6}(\mathbb R^3))} 
\\
&\lesssim   \|A_1^{1/2} u_0\|_{L^2(\mathbb R^3)}= \| u_0\|_V.
\end{align*}
When $\sigma>1$ we work with the admissible Strichartz pair 
$(2\sigma, \frac{6\sigma}{3\sigma-2})$ and get
\[
 \| I_1\|_{L^{2\sigma}(0,T; H^{1,\frac{6\sigma}{3\sigma-2} } (\mathbb R^3))}
 \lesssim   \|A_1^{1/2} u_0\|_{L^2(\mathbb R^3)}= \| u_0\|_V;
\]
since $\frac{6\sigma}{3\sigma-2}>3$  for $1<\sigma<2$,
we obtain the $L^\infty(\mathbb R^3)$-norm estimate.

The estimate for $I_2$ is more involved  and therefore we postpone it to  the Appendix
\ref{stime-I2-d3}.

It remains to estimate the term $ I_3$. For any $\sigma>0$ we use the H\"older inequality and the
stochastic Strichartz estimate \eqref{eqn-Strichartz-213}
for the admissible pair 
$(2+\frac{\sigma^2}2,6\frac{4+\sigma^2}{4+3\sigma^2})$; 
therefore
\[\begin{split}
\mathbb E \|I_3\|^{2\sigma}
_{L^{2\sigma}(0,T;H^{1,6\frac{4+\sigma^2}{4+3\sigma^2}}(\mathbb R^3))}
&\lesssim_T
\mathbb E \|I_3\|^{2\sigma}
_{L^{2+\frac{\sigma^2}2}(0,T;H^{1,6\frac{4+\sigma^2}{4+3\sigma^2}}(\mathbb R^3))}
\\&\lesssim
\|\Phi\|_{L_{HS}(U;V)}^{2\sigma}.
\end{split}\]
 \end{proof}

Notice that the restriction $\sigma<\frac{1+\sqrt {17}}4$ 
on the  power of the nonlinearity affects only the defocusing case.

We conclude this section by remarking that there is no similar result for $d\ge 4$. 
\begin{remark}
For larger dimension, there is no result similar to those in this section. Indeed, if one looks for 
$u\in L^{2\sigma}(0,T; H^{1,q}(\mathbb R^d))\subset L^{2\sigma}(0,T; L^\infty(\mathbb R^d))$ 
it is necessary that \[q>d\] in order to have $ H^{1,q}(\mathbb R^d)\subset L^\infty(\mathbb R^d)$.
Already the estimate for $I_1$ does not hold under this assumption. Indeed the  homogeneous 
Strichartz estimate \eqref{hom_Str} provides
\[
I_1 \in C([0,T];H^1(\mathbb R^d))\cap L^{2\sigma}(0,T;H^{1,q}(\mathbb R^d))
\]
if
\[
\frac 1\sigma=d\left(\frac 12 -\frac 1q\right)\qquad \text{ and } \qquad
2\le q\le \frac {2d}{d-2} .
\]
Since $\frac {2d}{d-2}\le 4$ for $d\ge 4$, 
the latter condition $q\le \frac {2d}{d-2}$  and the condition $q>d$ are incompatible for $d\ge 4$.

Let us notice that also in the deterministic setting the results on the attractors  are known for $d\le 3$, see \cite{Laur}.
\end{remark}

\section{The support of  the invariant measures}
\label{Sect-support}
Now we show some more properties on the invariant measures and their
 support. 
In dimension $d=2$ and $d=3$, thanks to the regularity results of Section \ref{Sect-regularity} 
we provide an estimate for the moments in the $V$ and $L^\infty(\mathbb R^d)$-norm.

\begin{proposition}\label{prop-media-p}
Let $d=2$ or $d=3$ and the Assumptions \ref{ass_G} and \ref{ass_lambda} hold.
\\
Let   $\mu$ be an invariant measure for equation \eqref{EQS}.
 Then for any finite $m\ge 1$ we have
 \begin{equation}\label{support-V-potenza-r}
\int \|x\|_{V}^{2m} \ d\mu(x)
\le  \phi_4^{m} \lambda^{-m}
\end{equation}
for some  smooth  function $\phi_4=\phi_4(d, \sigma,\lambda,m,\Phi)
    $ which is strictly decreasing w.r.t. $\lambda$.
    
Moreover,  supposing in addition  that $\sigma<\frac{1+\sqrt{17}}4$ when $d=3$, we have
\begin{equation}\label{supporto-Linfty-potenza-r}
\int \|x\|_{L^\infty}^{2\sigma} d\mu(x)\le   \phi_5(\lambda),
\end{equation}
where $\phi_5(\lambda)$ is a smooth decreasing function,
 depending also on the parameters $d, \sigma$ and on 
$\|\Phi\|_{L_{HS}(U;V)}$, 
\end{proposition}
\begin{proof}  
As far as \eqref{support-V-potenza-r} is concerned, we define the bounded mapping
$\Psi_k$ on $V$  as
\[
\Psi_k(x)=\begin{cases}
\|x\|_V^{2m}, & \text{if} \ \|x\|_V \le k
\\
k^{2m}, & \text{otherwise}
\end{cases}
\]
By the invariance of $\mu$ and the boundedness of $\Psi_k$, we have
\begin{equation}\label{invarianza}
\int_{V} \Psi_k \, {\rm d}\mu = \int_{V} P_s \Psi_k\, {\rm d}\mu\qquad \forall s>0.
\end{equation}
So
\[
P_s \Psi_k(x)=\mathbb E [\Psi_k(u(s;x))]
\le\mathbb E \|u(s;x)\|_V^{2m}.
\]
Moreover, from Corollary \ref{corollario-normaV} we get an estimate for 
$\mathbb E \|u(s;x)\|_V^{2m}$,
 and letting $s \to +\infty$ the first terms in \eqref{defocusing-V-2m} 
 and \eqref{focusing-V-2m} vanish so 
 we get
\[
\lim_{s\to+\infty} P_s \Psi_k(x)\le \phi_4^{m} \lambda^{-m}
\qquad \forall x\in V
\]
where 
\[
\phi_4=\begin{cases}\phi_1+\|\Phi\|_{L_{HS}(U;H)}^2, &\text{ for } \alpha=-1\\
\phi_2+\|\Phi\|_{L_{HS}(U;H)}^2, &\text{ for } \alpha=1\end{cases}
\] 
The same holds for the integral, that is 
\[
\lim_{s\to+\infty} \int_V P_s \Psi_k(x)\, {\rm d}\mu(x)
\le\phi_4^{m} \lambda^{-m}
\]
by the Bounded Convergence Theorem.
From \eqref{invarianza} we get
\[
 \int_V  \Psi_k\, {\rm d}\mu\le \phi_4^{m} \lambda^{-m}
 \]
as well.
Since $\Psi_k$ converges pointwise and monotonically from below to $\|\cdot\|_V^{2m}$, 
the Monotone Convergence Theorem yields \eqref{support-V-potenza-r}.

As far as \eqref{supporto-Linfty-potenza-r} is concerned, 
we consider the estimate \eqref{Goubet_est}  for $T=1$ and set $\tilde\Psi( u)= \|u\|_{L^\infty(\mathbb R^d)}^{2\sigma}$; this
defines a mapping
$\tilde\Psi:V \rightarrow \mathbb{R}_+ \cup \{+\infty\}$. 
Its approximation $\tilde\Psi_k:V\rightarrow \mathbb{R}_+ $, given by 
\begin{equation*}
\tilde\Psi_k(u)=
\begin{cases}
\|u\|^{2\sigma}_{L^\infty(\mathbb R^d)}, & \text{if} \ \|u\|_{L^\infty(\mathbb R^d)} \le k
\\
k^{2\sigma}, & \text{otherwise}
\end{cases}
\end{equation*}
defines a bounded mapping $\tilde\Psi_k:V\rightarrow \mathbb{R}_+ $.

It obviously holds
\begin{equation*}
\int_{V} \tilde\Psi_k\, {\rm d}\mu= \int_0^1 \left(\int_{V} \tilde\Psi_k \, {\rm d}\mu\right)\, {\rm d}s.
\end{equation*}
By the invariance of $\mu$ and the boundedness of $\tilde\Psi_k$, it also holds 
\begin{equation*}
\int_{V} \tilde\Psi_k \, {\rm d}\mu = \int_{V} P_s \tilde\Psi_k\, {\rm d}\mu\qquad \forall s>0.
\end{equation*}
Thus, by Fubini-Tonelli Theorem, since 
$\tilde\Psi_k(u)=\|u\|^{2\sigma}_{L^\infty(\mathbb R^d)} \wedge k^{2\sigma} 
\le \|u\|^{2\sigma}_{L^\infty(\mathbb R^d)}$, we get by Proposition \ref{prop-media-d2}
\begin{align*}
\int_{V} \tilde\Psi_k &\, {\rm d}\mu 
= \int_0^1 \int_{V} P_s \tilde\Psi_k\, {\rm d}\mu\, {\rm d}s =\int_V\int_0^1 \mathbb{E}\left[\tilde\Psi_k(u(s;x)) \right]\,{\rm d}s\, {\rm d}\mu(x)
\\
&\le\int_V \mathbb{E} \int_0^1 \|u(s;x)\|^{2\sigma}_{L^\infty(\mathbb R^d)} {\rm d}s\, {\rm d}
\mu(x)
\\&
 \le
 C \int_V \left( \|x\|_V^{2\sigma}
 + \lambda^{-b} 
  \psi(x)^{\sigma(2\sigma+1)}
  +\phi_3^{\sigma(2\sigma+1)}\lambda^{-{\sigma(2\sigma+1)}} 
  + \|\Phi\|_{L_{HS}(U;V)}^{2\sigma}\right) {\rm d}\mu(x)
\\& \phantom{ \|x\|_V^{2\sigma}
 + \lambda^{-b} 
  \psi(x)^{\sigma(2\sigma+1)}
  +\phi_3^{\sigma(2\sigma+1)}\lambda^{-{\sigma(2\sigma+1)}} }
\quad \text{  from } \eqref{Goubet_est}.
\end{align*}
Keeping in mind \eqref{psi-u-0} and   \eqref{support-V-potenza-r} 
we obtain the existence of a smooth decreasing function
$\phi_5(\lambda)$, depending also on the parameters $d, \sigma$ and on 
$\|\Phi\|_{L_{HS}(U;V)}$, such that
\[
\int_{V} \tilde\Psi_k\, {\rm d}\mu 
\le \phi_5(\lambda)
\]
for any $k$.

Since $\tilde\Psi_k$ converges pointwise and monotonically from below to $\tilde\Psi$, the Monotone Convergence Theorem yields the same bound for 
$
\int_{V} \|x\|^{2\sigma}_{L^\infty(\mathbb R^d)}\, {\rm d}\mu (x)$. This proves \eqref{supporto-Linfty-potenza-r}.
\end{proof}

\section{Uniqueness of the invariant measure for  sufficiently large damping}
\label{S-uniq}

We will prove that, if the damping coefficient $\lambda$ is sufficiently large, then the invariant measure is unique. 
\begin{theorem}
\label{uniq_thm}
Let $d=2$ or $d=3$. In addition to the Assumptions \ref{ass_G} and \ref{ass_lambda}
we suppose that $\sigma<\frac{1+\sqrt{17}}4$ when $d=3$.
\\
If 
\begin{equation}
\label{condition_beta}
\lambda>2  \phi_5(\lambda) 
\end{equation}
where  $\phi_5$ is the function appearing in Proposition \ref{prop-media-p},
then there exists a unique invariant measure for equation \eqref{EQS}.
\end{theorem}
\begin{proof} The existence of an invariant measure  comes from Theorem \ref{existence_inv_mea}.
Now we prove uniqueness by means of a reductio ad absurdum.
 Let us suppose that there exists more than one invariant measure. In particular there exist two different 
ergodic invariant measures $\mu_1$ and $\mu_2$. For both of them  Proposition \ref{prop-media-p}  holds. Fix either $i=1$ or $i=2$ and consider 
 any $f\in L^1(\mu_i)$. Then 
by the Birkhoff Ergodic Theorem (see, e.g., \cite{sinai})
 for $\mu_i$-a.e. $x_i\in V$ we have
\begin{equation}\label{Birk-f}
\lim_{t\to +\infty}\frac 1t \int_0^t f(u(s;x_i))\ ds=\int_V f\  d\mu_i \qquad
\mathbb P-a.s.
\end{equation}
Here $u(t;x)$ is the solution at time $t$, with initial value $u(0)=x\in V$.

Now fix two initial data $x_1$ and $x_2$ belonging, respectively, 
to the support of  the measure $\mu_1$ and $\mu_2$.
We have
\[
\int_V f \ d\mu_1- \int_V f \ d\mu_2=\lim_{t\to +\infty}\frac 1t \int_0^t [f(u(s;x_1))-f(u(s;x_2))]\ ds
\]
$\mathbb P$-a.s.. 
Taking any arbitrary $f$ in the set $ \mathcal G_0$ defined in \eqref{G0}, we get
\[
\left|\int_V f \ d\mu_1-\int_V f \ d\mu_2\right|
\le L
\lim_{t\to +\infty}\frac 1t \int_0^t  \|u(s;x_1)-u(s;x_2)\|_H ds .
\]
If we prove that 
\begin{equation}\label{convergenza-w}
\lim_{t\to +\infty}\|u(t;x_1)-u(t;x_2)\|_H =0 \qquad \mathbb P-a.s.,
\end{equation}
then  we conclude that 
 \[
 \int_V f \ d\mu_1-\int_V f \ d\mu_2=0
 \]
 so
 $\mu_1=\mu_2$ thanks to Lemma \ref{G-determining}.
So let us focus on the limit \eqref{convergenza-w}.

With a  short notation we write  $u_i(t)=u(t;x_i)$. 
Then consider the difference $w=u_1-u_2$ fulfilling
\[\begin{cases}
\frac{d}{dt} w(t)-\im A_0  w(t)+ \im F_\alpha(u_1(t))- \im F_\alpha(u_2(t)) +\gamma w (t)
= 0\\
w(0)=x_1-x_2
\end{cases}
\]
so
\[
\frac 12 \frac{d}{dt} \|w(t)\|_H^2 +\gamma \|w(t)\|_H^2
\le \int_{\mathbb R^d}  \Big| [ |u_1(t)|^{2\sigma} u_1(t)- |u_2(t)|^{2\sigma} u_2(t)]w(t) \Big| dy.
\]
Using  the elementary estimate 
\[
| |u_1|^{2\sigma} u_1- |u_2|^{2\sigma} u_2|
\le C_\sigma [|u_1|^{2\sigma}+|u_2|^{2\sigma}]|u_1-u_2|,
\]
we bound the nonlinear term in the r.h.s. as
\[
 \int_{\mathbb R^d} |[ |u_1|^{2\sigma} u_1- |u_2|^{2\sigma} u_2]w | \ dy
 \le
 [\|u_1\|^{2\sigma}_{L^\infty(\mathbb R^d)}+\|u_2\|^{2\sigma}_{L^\infty(\mathbb R^d)}] \|w\|_{L^2(\mathbb R^d)}^2.
\]
Therefore
\[
\frac{{\rm d}}{{\rm d}t} \|w(t)\|^2_H+ 2\lambda \|w(t)\|^2_H
\le
2\left( \|u_1(t)\|^{2\sigma}_{L^\infty(\mathbb R^d)}
      +\|u_2(t)\|^{2\sigma}_{L^\infty(\mathbb R^d)}\right) \|w(t)\|_{H}^2 .
\]
Gronwall inequality gives
\[
 \|w(t)\|^2_H \le  \|w(0)\|^2_H
 e^{-2\lambda t+2\int_0^t  \left( \|u_1(s)\|^{2\sigma}_{ L^\infty(\mathbb R^d)}+ \|u_2(s)\|^{2\sigma}_{ L^\infty(\mathbb R^d)}\right)ds}
\]
that is
\begin{equation}\label{stima-w}
 \|w(t)\|^2_H \le  \|x_1-x_2\|^2_H  e^{-2t \left[\lambda -\frac 1t \int_0^t  \left( \|u_1(s)\|^{2\sigma}_{ L^\infty(\mathbb R^d)}+ \|u_2(s)\|^{2\sigma}_{ L^\infty(\mathbb R^d)}\right)ds\right]}.
\end{equation}
This is a pathwise estimate.

We know from  Proposition \ref{prop-media-p} that $f(x)=\|x\|_{L^\infty(\mathbb R^d)}^{2\sigma}
\in L^1(\mu_i)$; therefore \eqref{Birk-f} becomes
\[
\lim_{t\to +\infty}\frac 1t \int_0^t \|u(s;x_i)\|^{2\sigma}_{L^\infty(\mathbb R^d)} ds
=
\int_V \|x\|^{2\sigma}_{L^\infty(\mathbb R^d)}\  d\mu_i (x)
\le \phi_5(\lambda)
\]
$\mathbb P$-a.s., for either $i=1$ or $i=2$.
Therefore, if 
\[
\lambda>2 \phi_5(\lambda),
\]
the exponential term in  the r.h.s. of \eqref{stima-w} vanishes as $t\to+\infty$. 
This proves  \eqref{convergenza-w} and  concludes the proof.
\end{proof}

\appendix
\section{Strichartz estimates}

In this section we recall the deterministic and stochastic Strichartz estimates on $\mathbb{R}^d$.

\begin{definition}\label{def-admissible}
We say that a pair  $(p,r)$ is admissible if 
\begin{equation*}
 \frac 2p+\frac dr=\frac d2 \qquad \text{and}  \quad (p,r) \ne (2, \infty)
\end{equation*}
and
\[\begin{cases}
2\le r\le \frac{2d}{d-2} &\text{ for } d\ge 3\\
2\le r<\infty &\text{ for } d=2\\
2\le r\le \infty &\text{ for } d=1\end{cases}
\]
\end{definition}
\noindent
If  $(p,r)$ is an admissible pair, then $2\le p\le \infty$.

Given $1\le \gamma \le \infty$, we denote by $\gamma^\prime$ its conjugate exponent, i.e.
$\frac 1\gamma+\frac1{\gamma^\prime}=1$.

\begin{lemma}
Let $(p,r)$ be an admissible pair of exponents. Then the following properties hold
\\
{\bf i)}
for every $\varphi \in L^2(\mathbb{R}^d)$ the function $t \mapsto e^{itA_0}\varphi$ belongs to $L^p(\mathbb{R};L^r(\mathbb{R}^d)) \cap C(\mathbb{R} ;L^2(\mathbb{R}^d))$.
Furthermore, there exists a constant $C$ such that 
\begin{equation}
\label{hom_Str}
\|e^{i \cdot A_0}\varphi\|_{L^p(\mathbb{R};L^r(\mathbb{R}^d))}
\le 
C\|\varphi\|_{L^2(\mathbb{R}^d)}, \qquad \forall \varphi \in L^2(\mathbb{R}^d).
\end{equation}
\\
{\bf ii)} Let $I$ be an interval of $\mathbb R$ and $0\in J =\overline I$. If $(\gamma,\rho)$ is an  admissible pair
and $f\in L^{\gamma^\prime}(I;L^{\rho^\prime}(\mathbb R ^d))$, then  the function
$t \mapsto G_f(t)=\int_0^t e^{i(t-s)A_0} f(s) ds$ belongs to 
$L^q(I;L^r(\mathbb{R}^d)) \cap C(J ;L^2(\mathbb{R}^d))$.
Furthermore, there exists a constant $C$, independent of $I$,  such that 
\begin{equation}
\label{hom_Str_f}
\|G_f\|_{L^p(I;L^r(\mathbb{R}^d)) }\le C \|f\|_{L^{\gamma^\prime}(I;L^{\rho^\prime}(\mathbb R ^d))} , \qquad \forall f \in L^{\gamma^\prime}(I;L^{\rho^\prime}(\mathbb R ^d)).
\end{equation}
\end{lemma}
\begin{proof}
See \cite[Proposition 2.3.3]{Cazenave}.
\end{proof}

\begin{lemma}[stochastic Strichartz estimate]
Let $(p,r)$ be an admissible pair. Then for any $a\in (1,\infty)$  and $T<\infty$
 there exists a constant $C$ such that
\begin{equation}
\label{eqn-Strichartz-213}
\left \Vert\int_0^{\cdot} e^{i(\cdot-s)A_0}\Psi(s) {\rm dW }(s)\right\Vert_{L^{a}(\Omega,L^p(0,T; L^r(\mathbb{R}^d)))}
\le C \|\Psi\|_{L^2(0,T;L_{HS}(U,L^2(\mathbb{R}^d)))}
\end{equation}
for any $\Psi \in L^2(0,T;L_{HS}(U,L^2(\mathbb{R}^d)))$.
\end{lemma}
\begin{proof}
See \cite[Proposition 2]{Horn}.
\end{proof}

\section{Determining sets}

The set
\begin{equation}
\mathcal G_1=\left\{f\in C_b(V):  \sup_{u\neq v} \frac{f(u)-f(v)}{\|u-v\|_V}<\infty   \right\}
\end{equation}
is a  
determining set for  measures on $V$ (see, e.g., \cite{Bil} Theorem 1.2).  This means that given two probability measures $\mu_1$ and $\mu_2$ on $V$ we have
\[
\int_V f \ d\mu_1=\int_V f \ d\mu_2 \qquad \forall f \in \mathcal G_1\quad
\Longrightarrow \mu_1=\mu_2.
\]
Following  Remark 2.2 in  \cite{GHMR2017}   we can consider as a determining set for 
measures on $V$ the set 
\begin{equation}\label{G0}
\mathcal G_0=\left\{f\in C_b(V):  \sup_{u\neq v} \frac{f(u)-f(v)}{\|u-v\|_H}<\infty   \right\}
\end{equation}
involving the weaker $H$-norm instead of the $V$-norm.
Indeed we have
\begin{lemma}\label{G-determining}
Let $\mu_1$ and $\mu_2$ be two invariant measures. 
If 
\[
\int_V f \ d\mu_1=\int_V f \ d\mu_2 \qquad \forall f \in \mathcal G_0 ,
\]
then $\mu_1=\mu_2$.
\end{lemma}
\begin{proof} We show the proof since we work in $\mathbb R^d$ whereas  
\cite{GHMR2017} deals with a bounded domain.

Set $P_Nx $ to be the element whose Fourier transform  is $1_{|\xi|\le N}\mathcal F(x)$; hence
$\|P_N x\|_V\le \sqrt{1+N^2}\|x\|_H$. Now we show that
any function $f\in \mathcal G_1$  
can be approximated by a function $f_N\in\mathcal G_0$ by setting $f_N(x)=f(P_N x)$. Indeed
\[
\left|f_N(x)-f_N(y)\right|
\le
L \|P_N x-P_N y\|_V
\le  L \sqrt{1+N^2} \| x- y\|_H
\]

By assumption we know that
\[
\int_V f_N \ d\mu_1=\int_V f_N \ d\mu_2
\]
Taking the limit as $N\to+\infty$, by the  bounded convergence theorem we get the same identity for $f \in \mathcal G_1$. Hence  $\mu_1=\mu_2$.
\end{proof}

\section{Estimate of the nonlinearity}
We consider $F(u)=|u|^{2\sigma}u$.  
\
\begin{lemma}\label{lemma_stimaF}
Let $d=2$.  For any $\sigma>0$, if   $p\in (1,2)$   is defined as
\begin{equation}\label{d2-F(u)}
 p=\begin{cases}\frac 2{2\sigma+1},&\qquad 0<\sigma<\frac 12\\
  \frac 43, &\qquad\sigma\ge \frac 12
 \end{cases}\end{equation}
then
\begin{equation}\label{stimaF_d2}
\|F(u)\|_{H^{1,p}(\mathbb R^2)}
\lesssim
 \|u\|_{ H^1(\mathbb R^2)}^{2\sigma+1} \qquad \forall u \in H^1(\mathbb R^2).
\end{equation}
Let $d=3$. For any $\sigma\in (0,\frac 32]$
 we have
\begin{equation}\label{stimaF_d3_1p}
  \|F(u)\|_{H^{1,\frac {6}{2 \sigma+3}}(\mathbb R^3)}
  \lesssim
   \|u\|_{ H^1(\mathbb R^3)}^{2\sigma+1} \qquad \forall u \in H^1(\mathbb R^3)
\end{equation}
and for any $\sigma \in [1,\frac 32]$ we have
\begin{equation}\label{stimaF_d3}
\|F(u)\|_{H^{2-\sigma,\frac 65}(\mathbb R^3)}
\lesssim
 \|u\|_{ H^1(\mathbb R^3)}^{2\sigma+1} \qquad \forall u \in H^1(\mathbb R^3).
\end{equation} 
\end{lemma}
\begin{proof} 
We start with the case $d=2$. To estimate the $H^{1,p}$-norm of $F$ it is enough to deal with 
$\|F\|_{L^p(\mathbb R^d)}$ 
and $\|\partial F\|_{L^p(\mathbb R^d)}$.
We compute
\begin{equation}
\partial F(u)= \sigma |u|^{2\sigma -2}\left(\bar u\partial u+u \partial \bar u\right)u +|u|^{2\sigma}\partial u, \qquad  \text{for an arbitrary} \ u \in V,
\end{equation}
and thus $|\partial F(u)| \lesssim_{\sigma}|u|^{2\sigma} |\partial u|$.

We have
\begin{align}
\label{stima_F}
\| F(u)\|_{L^p(\mathbb R^d)}
&= \|u\|_{L^{(2\sigma +1)p}(\mathbb R^d)}^{ 2\sigma+1}\\
\intertext{ and  the H\"older inequality, for $1\le p<2$, gives}
\label{stima_pa_F}
\|\partial F(u)\|_{L^p(\mathbb R^d)}
&\le
\||u|^{2\sigma}\|_{L^{\frac{2p}{2-p}}(\mathbb R^d)} \|\partial u\|_{L^2(\mathbb R^d)}
\\&\notag
\le  \|u\|_{L^{\frac{4 \sigma p}{2-p}}(\mathbb R^d)} ^{2\sigma}\|u\|_{V}
\end{align}
We recall the Sobolev embedding
 \[
 H^1(\mathbb R^2)\subset L^r(\mathbb R^2) \text{ for any } 2\le r<\infty .
 \]
 Therefore, if
\[  \begin{cases}
  2 \le (2\sigma +1)p  
  \\ 2 \le \dfrac{4 \sigma p}{2-p} 
   \end{cases}\]
 then both the r.h.s. of  \eqref{stima_F} and \eqref{stima_pa_F} can be estimated by a quantity 
 involving the  $H^1(\mathbb R^2)$-norm. The two latter inequalities are the same as
 \[
 p\ge \frac 2{2\sigma+1}
 \]
so one easily sees that the choice \eqref{d2-F(u)}
  allows to fulfil the two required estimates, i.e.
 $\| F(u)\|_{L^p(\mathbb R^d)}\lesssim \|u\|_V^{2\sigma+1}$ 
 and $\|\partial F(u)\|_{L^p(\mathbb R^d)}\lesssim \|u\|_V^{2\sigma+1}$.
 This proves \eqref{stimaF_d2}.

For $d=3$, first we show that for any $\sigma\in (0,\frac 32]$
 \begin{equation}\label{Fu-1p}
\|F(u)\|_{H^{1,p}(\mathbb R^3)}
\lesssim
 \|u\|_{ H^1(\mathbb R^3)}^{2\sigma+1} \qquad \forall u \in H^1(\mathbb R^3)
\end{equation}
with  $p=\frac {6}{2 \sigma+3}\in [1,2)$.

To this end we notice  that the r.h.s. of  \eqref{stima_F} and \eqref{stima_pa_F}
 are  estimated by a quantity involving the 
 $H^1(\mathbb R^3)$-norm if $ H^1(\mathbb R^3)\subset L^{(2\sigma +1)p}(\mathbb R^3)$ and 
 $H^1(\mathbb R^3)\subset L^{\frac{4 \sigma p}{2-p}}(\mathbb R^3)$.
 Recalling the  Sobolev embedding
 \[
 H^1(\mathbb R^3)\subset L^r(\mathbb R^3) \text{ for any } 2\le r\le 6
 \]
 we get the conditions
 \begin{align}
\label{s-p1} &2\le (2\sigma +1)p\le 6  \qquad \text{ (equivalent to }  \frac 2{2\sigma +1}\le p\le \frac 6{2\sigma +1})
 \\
 \label{s-p2}&2\le \frac{4 \sigma p}{2-p}\le 6 
  \qquad \text{ (equivalent to }  \frac 2{2\sigma +1}\le p\le \frac 6{2\sigma +3})
 \end{align}
Notice that \eqref{s-p2} is stronger than \eqref{s-p1}; moreover \eqref{s-p2} has a solution
 $p \in [1,2)$ only if $\sigma\in (0,\frac 32]$. 
 Choosing
  \begin{equation}\label{d3-F(u)}
 p=\frac {6}{2 \sigma+3}\in [1,2)
 \end{equation}
 we fulfil all the requirements and so we have  proved \eqref{stimaF_d3_1p}.
 
Now for  $1\le \sigma\le \frac 32$ there is  the continuous embedding 
$H^{1,\frac {6}{2 \sigma+3}}(\mathbb R^3)\subseteq H^{2-\sigma,\frac 65}(\mathbb R^3)$.
 Hence from \eqref{stimaF_d3_1p} we get \eqref{stimaF_d3}. 
  \end{proof}

\section{Computations in the proof of Proposition \ref{prop-media-d2} }
\subsection{From \eqref{stima-I2-d2} to \eqref{Goubet_est}} \label{s-conti}

From \eqref{stima-I2-d2} we proceed  as follows.
We distinguish different values of the parameter $\sigma$.
\\
$\bullet \sigma\in (0,\frac 14)$:  we have $\gamma^\prime=\frac1{1-\sigma}$, 
so $\frac {2\sigma} {\gamma^\prime}=2\sigma(1-\sigma)<\frac 38$ and 
$\gamma^\prime(2\sigma+1)=\frac{2\sigma+1}{1-\sigma}<2$. With the H\"older inequality twice
\[\begin{split}
\mathbb E\left( \int_0^T \|u(t)\|^{\gamma^\prime(2\sigma+1)}_{V}\ {\rm d}t\right)^{\frac {2\sigma} {\gamma^\prime}}
&\le
\left(\mathbb E\int_0^T  \|u(t)\|^{\gamma^\prime(2\sigma+1)}_{V}\ {\rm d}t\right)^{\frac {2\sigma}{\gamma^\prime}}
\\&\lesssim_T
\left(\mathbb E\int_0^T  \|u(t)\|^2_{V}\ {\rm d}t\right)^{\frac {2\sigma}{\gamma^\prime}\gamma^\prime\frac{2\sigma+1}2}
\end{split}
\]
 We conclude by means of the estimate of Corollary  \ref{corollario-normaV} for $m=1$; for instance in case {\bf i)}
 \[\begin{split}
& \left(\mathbb E\int_0^T  \|u(t)\|^2_{V}\ {\rm d}t\right)^{\sigma(2\sigma+1)}
 \\&\lesssim_{d,\sigma}
  \left(\int_0^T \Big[e^{- \lambda  t}  [{\mathcal{H}}(u_0)+{\mathcal{M}}(u_0)]
    + [\phi_1 + \|\Phi\|^{2}_{L_{HS}(U;H)}] \lambda^{-1}\Big] \ {\rm d}t\right)^{\sigma(2\sigma+1)}
\\&\lesssim_{d,\sigma,T}
  \left(   [{\mathcal{H}}(u_0)+{\mathcal{M}}(u_0)] \lambda^{-1} 
    + [\phi_1 + \|\Phi\|^{2}_{L_{HS}(U;H)}] \lambda^{-1}\right)^{\sigma(2\sigma+1)}
\\&\lesssim_{d,\sigma,T}
    [{\mathcal{H}}(u_0)+{\mathcal{M}}(u_0)]^{\sigma(2\sigma+1)} \lambda^{-\sigma(2\sigma+1)} 
    + [\phi_1 + \|\Phi\|^{2}_{L_{HS}(U;H)}]^{\sigma(2\sigma+1)} \lambda^{-\sigma(2\sigma+1)}
 \end{split}
\]
$\bullet \sigma\in [\frac 14, \frac 12)$:
 we have $\gamma^\prime=\frac1{1-\sigma}$, 
 so $\frac {2\sigma} {\gamma^\prime}=2\sigma(1-\sigma)\le\frac 12$ and 
$\gamma^\prime(2\sigma+1)=\frac{2\sigma+1}{1-\sigma}\ge 2$. With the H\"older inequality
\[
\mathbb E\left( \int_0^T \|u(t)\|^{\gamma^\prime(2\sigma+1)}_{V}\ {\rm d}t\right)^{\frac {2\sigma} {\gamma^\prime}}
\le
\left(\mathbb E\int_0^T  \|u(t)\|^{\gamma^\prime(2\sigma+1)}_{V}\ {\rm d}t\right)^{\frac {2\sigma}{\gamma^\prime}}
\]  
and then we conclude by means of the estimate of Corollary  \ref{corollario-normaV} for $2m=\gamma^\prime(2\sigma+1)$; for instance in case {\bf ii)}
 \[\begin{split}
&\left(\mathbb E\int_0^T  \|u(t)\|^{\gamma^\prime(2\sigma+1)}_{V}\ {\rm d}t\right)^{\frac {2\sigma}{\gamma^\prime}}
\\&\lesssim
\Big(  \int_0^T e^{- a \frac {\gamma^\prime}2(2\sigma+1)\lambda  t} dt \
[\tilde {\mathcal{H}}(u_0)^{ \frac {\gamma^\prime}2(2\sigma+1)} 
+ (\lambda^{- \frac {\gamma^\prime}2(2\sigma+1)} +\lambda^{-\frac{ \frac {\gamma^\prime}2(2\sigma+1)-1}2})
\mathcal{M}(u_0)^{ \frac {\gamma^\prime}2(2\sigma+1)(1+\frac{2\sigma}{2-\sigma d})}
\\&
\quad    +{\mathcal{M}}(u_0)^{ \frac {\gamma^\prime}2(2\sigma+1)}]
+ T[\phi_2 + \|\Phi\|^{2}_{L_{HS}(U;H)}]^{ \frac {\gamma^\prime}2(2\sigma+1)} \lambda^{- \frac {\gamma^\prime}2(2\sigma+1)} \Big)^{\frac {2\sigma}{\gamma^\prime}}
\\&
\lesssim_{a,\sigma,T} \lambda^{-2\sigma(1-\sigma)}
\big[\tilde {\mathcal{H}}(u_0) +(\lambda^{-1}+\lambda^{-\frac{4\sigma-1}{4\sigma(1-\sigma)(2\sigma+1)}} ) \mathcal{M}(u_0)^{1+\frac{2\sigma}{2-\sigma d}}+  \mathcal{M}(u_0)\big]^{\sigma(2\sigma+1)}
\\&\quad
+[\phi_2 + \|\Phi\|^{2}_{L_{HS}(U;H)}]^{\sigma(2\sigma+1)} \lambda^{- \sigma(2\sigma+1)}
\end{split}
\]
$\bullet \sigma\in [\frac 12,\frac 23)$:
 we have $\gamma^\prime=\frac 43$, so  $\frac {2\sigma} {\gamma^\prime}=\frac 32 \sigma<1$ and 
$\gamma^\prime(2\sigma+1)=\frac 43(2\sigma+1) \ge \frac 83$.  So we proceed as in the previous case.
\\
$\bullet \sigma\ge \frac 23$: we have $\gamma^\prime=\frac 43$, so  $\frac {2\sigma} {\gamma^\prime}\ge 1$ and 
$\gamma^\prime(2\sigma+1)\ge \frac {28}9$.  With the H\"older inequality
\[
\mathbb E\left( \int_0^T \|u(t)\|^{\gamma^\prime(2\sigma+1)}_{V}\ {\rm d}t\right)^{\frac {2\sigma} {\gamma^\prime}}
\lesssim_T
\mathbb E\int_0^T \|u(t)\|^{\gamma^\prime(2\sigma+1)\frac {2\sigma} {\gamma^\prime}}_{V}\ {\rm d}t 
\] 
and then we conclude by means of the estimate of Corollary  \ref{corollario-normaV} for 
$m=2\sigma(2\sigma+1)$.

\subsection{Estimate of $I_2$ when $d=3$}\label{stime-I2-d3}
We distinguish two ranges of values for $\sigma$.

$\bullet$  For $0<\sigma\le 1$ we have 
 $L^2(0,T;L^6(\mathbb R^3))\subseteq L^{2\sigma}(0,T;L^6(\mathbb R^3))$. 
 So we consider the admissible  Strichartz pair  $(2,6)$ and get for any 
 admissible  Strichartz pair $(\gamma,r)$
\[\begin{split}
\|I_2\|_{L^{2\sigma}(0,T;H^{1,6}(\mathbb R^3))}
&\lesssim
\|I_2\|_{L^{2}(0,T;H^{1,6}(\mathbb R^3))}
\\&=
\|A_1^{1/2} I_2\|_{L^{2}(0,T;L^{6}(\mathbb R^3))}
\\&\lesssim
\|A_1^{1/2} F_\alpha(u)\|_{L^{\gamma^\prime}(0,T;L^{r^\prime}(\mathbb R^3))} 
\quad\text{ by } \eqref{hom_Str_f}
\\&\lesssim
\| F_\alpha(u)\|_{L^{\gamma^\prime}(0,T;H^{1,r^\prime}(\mathbb R^3))} 
\end{split}\]
The parameters are such that  $\gamma^\prime=\frac{4r^\prime}{7r^\prime-6}$.
From the Definition \ref{def-admissible} we have the condition $2\le r\le 6$, equivalent to 
 $\frac 65\le r^\prime\le 2$.
 Choosing
 \[
 r^\prime=\frac 6{3+2\sigma},
 \]
 we have $ r^\prime\in  [\tfrac 65,2)$ when $0<\sigma\le 1$ 
 and $\gamma^\prime=\frac 2{2-\sigma}\in (1,2]$; thus 
 we can use  \eqref{stimaF_d3_1p} to estimate the nonlinearity $F_\alpha(u)$.
 Summing up, we have
 \[
 \|I_2\|_{L^{2\sigma}(0,T;L^{\infty}(\mathbb R^3))}
 \lesssim
 \|I_2\|_{L^{2\sigma}(0,T;H^{1,6}(\mathbb R^3))}
 \lesssim
 \|u\|^{2\sigma+1}_{L^{2 \frac{2\sigma+1}{2-\sigma}}(0,T;V)}.
 \]
 Hence
 \[\begin{split}
 \mathbb E  \|I_2\|^{2\sigma}_{L^{2\sigma}(0,T;L^{\infty}(\mathbb R^3))}
 &\lesssim
  \mathbb E \left(\int_0^T \|u(t)\|_{V}^{\frac2{2-\sigma}(2\sigma+1)} dt\right)^{\sigma(2-\sigma)}
 \\&\lesssim
\left(  \mathbb E \int_0^T \|u(t)\|_{V}^{2\frac{2\sigma+1}{2-\sigma}} dt\right)^{\sigma(2-\sigma)}
 \end{split}\]
 by H\"older inequality since $\sigma(2-\sigma)\le 1$.
 \\
 From here, bearing in mind Corollary \ref{corollario-normaV} we conclude as in the previous 
 subsection and we obtain the second and third terms in the r.h.s. of \eqref{Goubet_est}.

 $\bullet$ For $ \sigma>1$ we use the  admissible  Strichartz pair 
 $(2\sigma,\frac{6\sigma}{3\sigma-2})$
 so
 \[\begin{split}
\|I_2\|_{L^{2\sigma}(0,T;H^{\theta,\frac{6\sigma}{3\sigma-2}}(\mathbb R^3))}
&=
\|A_1^{\theta/2} I_2\|_{L^{2\sigma}(0,T;L^{\frac{6\sigma}{3\sigma-2}}(\mathbb R^3))}
\\&\lesssim
\|A_1^{\theta/2} F_\alpha(u)\|_{L^{2}(0,T;L^{\frac 65}(\mathbb R^3))} 
\\&\lesssim
\| F_\alpha(u)\|_{L^{2}(0,T;H^{\theta,\frac 65}(\mathbb R^3))} 
\end{split}\]
where we used   \eqref{hom_Str_f} with $\gamma^\prime=2$ and $\rho^\prime=\frac 65$, corresponding to the 
admissible  Strichartz pair $(\gamma,\rho)$ with $\gamma=2$ and $\rho=6$. Notice that 
$\frac 65$ is the minimal allowed value for $\rho^\prime$ when $d=3$.
\\
Now, assuming $1<\sigma\le \frac 32$ we use the estimate \eqref{stimaF_d3} with $\theta=2-\sigma$.
Summing up, we  obtain
\[
\|I_2\|_{L^{2\sigma}(0,T;H^{2-\sigma,\frac{6\sigma}{3\sigma-2}}(\mathbb R^3))}
\lesssim
\|u\|^{2\sigma+1}_{L^{2(2\sigma+1)}(0,T;H^1(\mathbb R^3))}.
\]
 When 
 \[
 (2-\sigma)\frac{6\sigma}{3\sigma-2}>3
 \]
 we have $H^{2-\sigma,\frac{6\sigma}{3\sigma-2}}(\mathbb R^3)\subset
 L^{\infty}(\mathbb R^3)$. This gives the condition $ \sigma<\frac{1+\sqrt {17}}4$. Hence
 \[
 \|I_2\|_{L^{2\sigma}(0,T;L^\infty(\mathbb R^3))}
\lesssim
\|u\|^{2\sigma+1}_{L^{2(2\sigma+1)}(0,T;V)}.
 \]
Since $\sigma>1$, we conclude with the H\"older inequality that
 \[\begin{split}
 \mathbb E  \|I_2\|^{2\sigma}_{L^{2\sigma}(0,T;L^{\infty}(\mathbb R^3))}
 &\lesssim
   \mathbb E \left(\int_0^T \|u(t)\|_{V}^{2(2\sigma+1)} dt\right)^{\sigma}
 \\&
 \lesssim_T \mathbb E\int_0^T \|u(t)\|_{V}^{2\sigma(2\sigma+1)} dt
 \end{split}\]
 Finally we obtain the second and third term of \eqref{Goubet_est}  
 by means of Corollary \ref{corollario-normaV} as before.

\section*{Acknowledgements}

B. Ferrario  and M. Zanella gratefully acknowledges financial support from GNAMPA-INdAM.


\end{document}